\documentclass{article}
\usepackage[utf8]{inputenc}

\usepackage{color}

\usepackage{amsfonts}
\usepackage{amssymb}
\usepackage{amsmath}
\usepackage{amsthm}
\usepackage{mathtools}
\usepackage{cleveref}
\usepackage{url}
\usepackage{tikz}
\usetikzlibrary{arrows}
\usepackage{tikz-cd}
\usepackage{adjustbox}
\usepackage{graphicx}
\usepackage{afterpage}
\usepackage[pass]{geometry}
\usepackage{authblk}
\usepackage{longtable}

\usepackage[
    backend=biber,
    style=alphabetic,
    sorting=nyt,
    doi=false,
    url=false,
    isbn=false,
    giveninits=true,
    maxbibnames=99 
]{biblatex}

\numberwithin{equation}{section}

\theoremstyle{plain}
    \newtheorem{theorem}{Theorem}
    \newtheorem{proposition}{Proposition}[section]
    
    \newtheorem{lemma}[proposition]{Lemma}

\theoremstyle{definition}
    \newtheorem{definition}[proposition]{Definition}

    \newtheorem{question}[proposition]{Question}
    
\theoremstyle{remark}
	\newtheorem{remark}[proposition]{Remark}

\newcommand{\ZZ}{\mathbb{Z}}
\newcommand{\QQ}{\mathbb{Q}}

\renewcommand{\H}{\mathcal{H}}

\newcommand{\isom}{\cong}

\title{
    Computations of HOMFLY homology
}
\author{Keita Nakagane, Taketo Sano}

\newcommand{\addresses}{{
  \noindent
  Keita Nakagane, \textsc{}\\%
  \textit{E-mail address}: \url{keita.nakagane@gmail.com}
  
  \bigskip

  \noindent
  Taketo Sano, \textsc{RIKEN iTHEMS, Wako, Saitama 351-0198, Japan }\\
  \textit{E-mail address}: \url{taketo.sano@riken.jp}
}}

\begin{document}

    \maketitle
    \begin{abstract}
    Khovanov--Rozansky's HOMFLY homology is determined for all prime knots with up to 11 crossings by direct computations. 
\end{abstract}
    
    
    \section{Introduction} \label{sec:intro}

The \textit{HOMFLY polynomial}\footnote{The definitions of the HOMFLY polynomial and the HOMFLY homology follows \cite{Ras15}.} \cite{HOMFLY, Jones, PT} $P(L)(q, a)$ is an invariant of oriented links $L$ in $S^3$, defined by the skein relation
\[
    a P(
    \begin{tikzpicture}[baseline=1]
        \draw[->] (0.3,0)--(0,0.3);
        \draw (0,0)--(0.1,0.1);
        \draw[->] (0.2,0.2)--(0.3,0.3);
    \end{tikzpicture}
    )
    -a^{-1}P(
    \begin{tikzpicture}[baseline=1]
        \draw[->] (0,0)--(0.3,0.3);
        \draw (0.3,0)--(0.2,0.1);
        \draw[->] (0.1,0.2)--(0,0.3);
    \end{tikzpicture}
    )
    =(q - q^{-1})P(
    \begin{tikzpicture}[baseline=1]
        \draw [bend right = 45,->] (0,0) to (0,0.3);
        \draw [bend left = 45,->] (0.3,0) to (0.3,0.3);
    \end{tikzpicture}
    )
\]
and the normalization $P(\text{trivial knot}) = 1$.
The invariant specializes to the Alexander polynomial by $\Delta(L)(q) = P(L)(q, 1)$ and to the $sl(N)$ polynomials by $P_N(L)(q) = P(L)(q, q^N)$ for $N > 0$. In particular, $N = 2$ gives the Jones polynomial.

Khovanov and Rozansky \cite{KR08} categorified the HOMFLY polynomial as a triply graded homology group $\H^{*,*,*}(L)$, called the (reduced) \textit{HOMFLY homology}, following Khovanov's categorification of the Jones polynomial \cite{Kho00} and Khovanov--Rozansky's categorification of the $sl(N)$ polynomial \cite{KR08-slN}. It's graded Euler characteristic gives the HOMFLY polynomial
\[
    P(L)(q,a) = \sum_{i,j,k} (-1)^{(k-j)/2} q^ia^j\dim \H^{i,\,j,\,k}(L).
\]

Although HOMFLY homology is defined combinatorially, due to the complexity of the construction, the homology groups have been determined only for some specific classes of knots and links. In \cite{Ras07}, Rasmussen showed that the homology of two-bridge knots are determined by the HOMFLY polynomial and the signature. In \cite{HM19}, Mellit and Hogancamp gave combinatorial formulas to compute the homology of torus links. As for direct computations, Rasmussen \cite{Ras15} used Webster's computer program \cite{Web} and determined the homologies for all prime knots with up to $9$ crossings.

In this paper, we give an algorithm to compute the HOMFLY homology of knots, and with the implementation of the algorithm we have determined the homologies for all prime knots with up to 11 crossings. The program is available at \cite{kr-calc}, together with all of the computation results in both human-readble and machine-friendly formats. Here we compress the results by listing only \textit{KR-thick knots}. Recall that the \textit{$\Delta$-grading} for $\H^{i,j,k}$ is defined by $\Delta = i + j + k$, and a knot $K$ is \textit{KR-thin} if $\H^{i,j,k}(K)$ is supported in a single $\Delta$-grading, \textit{KR-thick} otherwise. It is known that the HOMFLY homology for a KR-thin knot is determined by its signature and its HOMFLY polynomial, hence KR-thick knots are the ones of interest. 

\begin{theorem}
    Up to $11$ crossings, the following $96$ are the only KR-thick prime knots, all of which are supported in two $\Delta$-gradings. (The actual results are listed in \Cref{sec:results}.) 
    \begin{center}
    \resizebox{0.98\columnwidth}{!}{%
    \begin{tabular}{llllllllll}
        $8_{19}$ & $9_{42}$ & $9_{43}$ & $9_{47}$ & $10_{124}$ & 
        $10_{125}$ & $10_{128}$ & $10_{132}$ & $10_{134}$ & $10_{136}$ \\ 
        $10_{138}$ & $10_{139}$ & $10_{142}$ & $10_{145}$ & $10_{152}$ & 
        $10_{153}$ & $10_{154}$ & $10_{160}$ & $10_{161}$ & $11a_{263}$ \\ 
        $11n_{2}$ & $11n_{6}$ & $11n_{9}$ & $11n_{12}$ & $11n_{13}$ & 
        $11n_{16}$ & $11n_{19}$ & $11n_{20}$ & $11n_{23}$ & $11n_{24}$ \\ 
        $11n_{27}$ & $11n_{30}$ & $11n_{31}$ & $11n_{34}$ & $11n_{36}$ & 
        $11n_{38}$ & $11n_{39}$ & $11n_{41}$ & $11n_{42}$ & $11n_{44}$ \\
        $11n_{45}$ & $11n_{47}$ & $11n_{49}$ & $11n_{57}$ & $11n_{60}$ & 
        $11n_{61}$ & $11n_{64}$ & $11n_{67}$ & $11n_{69}$ & $11n_{70}$ \\ 
        $11n_{73}$ & $11n_{74}$ & $11n_{76}$ & $11n_{77}$ & $11n_{78}$ & 
        $11n_{79}$ & $11n_{80}$ & $11n_{81}$ & $11n_{82}$ & $11n_{88}$ \\
        $11n_{90}$ & $11n_{92}$ & $11n_{93}$ & $11n_{95}$ & $11n_{96}$ & 
        $11n_{97}$ & $11n_{102}$ & $11n_{104}$ & $11n_{107}$ & $11n_{111}$ \\ 
        $11n_{116}$ & $11n_{118}$ & $11n_{120}$ & $11n_{126}$ & $11n_{133}$ &
        $11n_{135}$ & $11n_{138}$ & $11n_{143}$ & $11n_{145}$ & $11n_{147}$ \\ 
        $11n_{148}$ & $11n_{149}$ & $11n_{150}$ & $11n_{151}$ & $11n_{152}$ & 
        $11n_{153}$ & $11n_{155}$ & $11n_{158}$ & $11n_{161}$ & $11n_{166}$ \\ 
        $11n_{169}$ & $11n_{173}$ & $11n_{175}$ & $11n_{177}$ & $11n_{182}$ & 
        $11n_{183}$
    \end{tabular}
    }
    \end{center}
\end{theorem}

Note that there is one alternating knot $11a_{263}$ that is KR-thick, as pointed out in \cite{Ras07}. This is in contrast to the fact that alternating knots are necessarily homologically thin for Khovanov homology \cite{Lee:2005} and for knot Floer homology \cite{OS:2003}. More observations on the results are given in \Cref{sec:observations}. We also suggest Chandler and Gorsky's paper \cite{CG22} where they use our results to further study the structures of the HOMFLY homologies, and computes the $S$-invariant and compares it to the $sl(N)$ concordance invariants.

This paper is organized as follows. In \Cref{sec:prelim}, we review the definition of the reduced HOMFLY homology $\H(L)$. In \Cref{sec:algo} we describe the proposed algorithm, and in \Cref{sec:comptuations} we discuss the actual computation. Finally in \Cref{sec:observations} we give more observations on the computational results. 

\subsection*{Acknowledgements} The authors thank Eugene Gorsky for helpful discussions. The authors also thank Hidetoshi Masai for organizing \textit{ToKoDai Topology Seminar} from which this joint work started. 
We also thank Katsumi Kishikawa and Toru3 for helping us developing the computer program, and Takumi Doi for helping us running the program on a high performance computer. 
The work of TS was supported by JSPS KAKENHI Grant Numbers 21J15094, 23K12982, RIKEN iTHEMS Program and academist crowdfunding.

    \section{Preliminaries} \label{sec:prelim}

Here we review the definition of the reduced HOMFLY homology based on Rasmussen's formulation \cite{Ras15}. Since we only consider the reduced theory, the term reduced will be omitted hereafter. 

Let $D$ be an oriented link diagram with $n$ crossings. For simplicity, we assume that $D$ has at least one crossing and that $D$ is connected as a planar diagram. 
The \textit{edge ring} $R(D)$ of $D$ is defined as follows. 
Let $G(D)$ be the underlying oriented $4$-valent graph of $D$ with edges labeled $1, \ldots, 2n$. Let $R'(D) = \QQ[X_1, \ldots, X_{2n}]$ endowed with a $\ZZ$-grading such that each $X_i$ has degree $2$. This grading is called the \textit{$q$-grading} on $R'(D)$. 
For each crossing $c$ of $D$, define
\[
    \rho(c) = X_k + X_l - X_i - X_j \in R'(D),
\]
where $i, j$ are the indices of the incoming edges at $c$, and $k, l$ are the ones for the outgoing edges (see \Cref{complex}). Also define
\[
    \theta = \sum_i X_i \in R'(D).
\]
Let $I(D)$ be the ideal of $R'(D)$ generated by $\theta$ and $\rho(c)$ where $c$ runs over all crossings of $D$. The \textit{edge ring} $R(D)$ is defined by the quotient ring $R'(D)/I(D)$, endowed with the $q$-grading inherited from $R'(D)$.

\begin{figure}[t]
    \begin{center}
    \begin{tikzpicture}[auto]
    \node (a) at (0, 0) {$R\{2,-2,-2\}$};
    \node (b) at (0,2.5) {$R\{0,-2,0\}$};
    \node (c) at (5,0) {$R\{0,0,-2\}$};
    \node (d) at (5,2.5) {$R\{0,0,0\}$};
    \node (C) at (-2.5, 1.25) {$C(D_c)=$};
    \draw[-latex, thick] (a) to node {$X_l-X_i$} (b);
    \draw[-latex, thick] (a) to node[swap] {$ (X_k-X_i)(X_l-X_i)$} (c);
    \draw[-latex, thick] (b) to node {$X_k-X_i$} (d);
    \draw[-latex, thick] (c) to node[swap] {$1$} (d);
    \node (p) at (7.75, 0.8) {$c$};
    \node (v1) at (7,2) {$k$};
    \node (v2) at (8.5,2) {$l$};
    \node (v3) at (7,0.5) {$i$};
    \node (v4) at (8.5,0.5) {$j$};
    \node (center) at (7.75,1.25) {};
    \draw [->, thick] (v3) edge (v2);
    \draw [->, thick] (center) edge (v1);
    \draw [thick] (v4) edge (center);
\end{tikzpicture}

\vspace{12pt}

\begin{tikzpicture}[auto]
    \node (a) at (-0.4, 0) {$R\{0,-2,0\}$};
    \node (b) at (-0.4,2.5) {$R\{0,-2,2\}$};
    \node (c) at (4.6,0) {$R\{0,0,0\}$};
    \node (d) at (4.6,2.5) {$R\{-2,0,2\}$};
    \node (C) at (-2.5, 1.25) {$C(D_c)=$};
    \draw[-latex, thick] (a) to node {$1$} (b);
    \draw[-latex, thick] (a) to node[swap] {$X_k-X_i$} (c);
    \draw[-latex, thick] (b) to node {$ (X_k-X_i)(X_l-X_i)$} (d);
    \draw[-latex, thick] (c) to node[swap] {$X_l-X_i$} (d);
    \node (p) at (7.75, 0.8) {$c$};
    \node (v1) at (7,2) {$k$};
    \node (v2) at (8.5,2) {$l$};
    \node (v3) at (7,0.5) {$i$};
    \node (v4) at (8.5,0.5) {$j$};
    \node (center) at (7.75,1.25) {};
    \draw [->, thick] (v4) edge (v1);
    \draw [->, thick] (center) edge (v2);
    \draw [thick] (v3) edge (center);
\end{tikzpicture}
    \end{center}
    \caption{The definition of $C(D_c)$}\label{complex}
\end{figure}

Next we define a chain complex $C(D)$ over $R(D)$. First for each crossing $c$ in $D$, let $D_c$ denote the subdiagram composed of the four edges of $D$ adjacent to $c$, and associate to it a double complex $C(D_c)$. The definition of $C(D_c)$ depends on the sign of $c$, as illustrated in \Cref{complex}. As a module, it is a direct sum of four copies of $R(D)$ endowed with a triple grading, where the first grading corresponds to the $q$-grading, the second and the third gradings are the $\ZZ^2$-homological gradings called the \textit{horizontal} and the \textit{vertical} gradings respectively. $R\{i, j, k\}$ denotes the copy of $R(D)$ with the triple grading shifted accordingly, so that $1 \in R$ has $q$-grading $i$ and homological grading $(j, k)$.  The two horizontal arrows give the \textit{horizontal differential} $d_H$ and the two vertical arrows give the \textit{vertical differential} $d_V$ on $C(D)$. One can see that $d_H$ and $d_V$ have degrees $(2,2,0)$ and $(0,0,2)$ respectively, and that they commute. Now define 
\[
    C(D) = \bigotimes_c C(D_c)
\]
where the tensor product is taken over $R(D)$ and $c$ runs over all crossings of $D$. Consider the homology group
\[
    \H'(D) = H(H(C(D), d_H), (d_V)_*)
\]
which reads, first take the homology of $C(D)$ with respect to the horizontal differential $d_H$, and then take the homology of $H(C(D), d_H)$ with respect to the induced differential $(d_V)_*$. We regard it as a triply graded $\QQ$-vector space, where $\H'^{i,j,k}(D)$ is the homology group of the $\QQ$-subcomplex of $C(D)$ with homogeneous $q$-grading $i$ at homological grading $(j, k)$. Finally define
\begin{equation} \label{eq:Hbar-def}
    \H(D) = \H'(D)\{-w + b - 1,\ w + b - 1,\ w - b + 1\},
\end{equation}
where $w$ is the writhe of $D$ and $b$ is number of Seifert circles of $D$. When $D$ is a closed braid diagram, $b$ is the number of strands.

\begin{theorem}[\cite{KR08, Ras15}]
    Let $L$ be an oriented link with a closed braid diagram $D$. The isomorphism class of $\H(D)$ as a triply graded $\QQ$-vector space is independent of the choice of $D$.
\end{theorem}

Thus the following definition is justified.

\begin{definition}
    For an oriented link $L$ with closed braid diagram $D$, the homology group $\H(D)$ is denoted $\H(L)$ and is called the \textit{reduced HOMFLY homology} of $L$.
\end{definition}

\begin{remark}\label{rem:generaldiag}
    The definition of $\H(D)$ is valid for general link diagrams $D$, but it is necessary that $D$ is chosen to be a braid closure for its invariance. More specifically, it is known that $\H(D)$ is not invariant under the RIIb move \cite{Abel17, Nak20}. Further considerations are given in \Cref{sec:observations}. 
\end{remark}

The following theorem states that the graded Euler characteristic of the HOMFLY homology $\H(L)$ gives the HOMFLY polynomial $P(L)$. 

\begin{theorem}[\cite{KR08, Ras15}]
    For any oriented link $L$, we have
    \begin{equation*}
        P(L)(q,a) = \sum_{i,j,k} (-1)^{(k-j)/2} q^ia^j\dim \H^{i,\,j,\,k}(L).
    \end{equation*}
\end{theorem}

Next we recall several structural results for HOMFLY homology, which assures the finiteness of our algorithm. The following is a categorified version of the \textit{Morton-Franks-Williams Inequality} \cite{Mor86, FW87}.
\begin{proposition}[{\cite{Kho07, Wu08}}] \label{prop:MFW-ineq}
    For a link $L$, its HOMFLY homology $\H^{*, j, *}(L)$ is supported in
    \[
        j \in [w - b + 1,\ w + b - 1].
    \]
\end{proposition}

The following is a categorified version of the symmetry of the HOMFLY polynomial $P_K(q,a) = P_K(-q^{-1}, a)$.

\begin{proposition}[{\cite{DGR,OR20,gorsky2021:tautological}}] \label{prop:q-sym}
    For a knot $K$, we have
    \[
        \H^{i, j, k}(K) \isom \H^{-i, j, k + 2i}(K).
    \]
\end{proposition}

We also have the duality with respect to mirroring.

\begin{proposition}[\cite{EMAK}] \label{prop:duality}
    For a knot $K$ and its mirror $m(K)$,
    \[
        \H^{i, j, k}(m(K)) \isom \H^{-i, -j, -k}(K).
    \]
\end{proposition}

Now we obtain a categorified version of the \textit{Morton bound} \cite{Mor86}.

\begin{proposition} \label{prop:morton}
    For a knot $K$, its HOMFLY homology $\H^{i, *, k}(K)$ is supported in
    \[
        i, k \in [-n + b - 1,\ n - b + 1].
    \]
\end{proposition}

\begin{proof}
    Take any diagram $D$ of $K$ and let $n^-$ be the number of negative crossings of $D$. From \Cref{complex}, we see that the lowest $q$-degree of the non-trivial subcomplex of $C(D)$ is $-2n^-$. Adding to it the $q$-grading shift $-w + b - 1$ gives the lower bound for $i$. The upper bound follows by symmetry. Similarly, the upper bound for $k$ can be seen from from \Cref{complex}, and the lower bound follows by duality. 
\end{proof}
    \section{Algorithm} \label{sec:algo}

The idea of the algorithm is similar to the one given in \cite{sano-sato:2020}. Suppose $F$ is a field, $R$ a multivariate polynomial ring over $F$, and $C$ a freely and finitely generated chain complex over $R$ whose differential preserves the polynomial grading. We may regard $C$ as a chain complex over $F$ and decompose it into (infinitely many) finitely generated homogeneous subcomplexes. Provided that these subcomplexes are acyclic outside some computable bound, the homology group $H(C)$ as a graded $F$-vector space can be fully computed algorithmically.

\subsection{Free variables of the edge ring}

Let $D$ be an oriented knot diagram with $n$ crossings.

\begin{lemma}[{\cite[Lemma 2.4]{Ras15}}]\label{dimlem}
    The edge ring $R(D)$ is isomorphic to a polynomial ring with $n$ variables.
\end{lemma}

\begin{proof}
    Recall that $R'(D)$ is generated by $2n$ indeterminates, and that the ideal $I(D)$ is generated by $n + 1$ linear polynomials. From the assumption that $D$ is connected, these polynomials satisfy a unique linear relation $\sum_c \rho(c) = 0$ .
\end{proof}

\begin{figure}
    \centering
    \tikzset{every picture/.style={line width=0.75pt}} 

\begin{tikzpicture}[x=0.75pt,y=0.75pt,yscale=-1,xscale=1]

\draw [color={rgb, 255:red, 128; green, 128; blue, 128 }  ,draw opacity=1 ][line width=1.5]    (39.62,19.12) -- (107.5,87) ;
\draw [shift={(37.5,17)}, rotate = 45] [color={rgb, 255:red, 128; green, 128; blue, 128 }  ,draw opacity=1 ][line width=1.5]    (14.21,-6.37) .. controls (9.04,-2.99) and (4.3,-0.87) .. (0,0) .. controls (4.3,0.87) and (9.04,2.99) .. (14.21,6.37)   ;
\draw [color={rgb, 255:red, 0; green, 0; blue, 0 }  ,draw opacity=1 ][line width=1.5]    (105.86,19.11) -- (37,87) ;
\draw [shift={(108,17)}, rotate = 135.41] [color={rgb, 255:red, 0; green, 0; blue, 0 }  ,draw opacity=1 ][line width=1.5]    (14.21,-6.37) .. controls (9.04,-2.99) and (4.3,-0.87) .. (0,0) .. controls (4.3,0.87) and (9.04,2.99) .. (14.21,6.37)   ;
\draw [color={rgb, 255:red, 0; green, 0; blue, 0 }  ,draw opacity=1 ][line width=1.5]    (203,87) .. controls (242.2,51.72) and (238.18,49.09) .. (205.53,18.88) ;
\draw [shift={(203.5,17)}, rotate = 402.85] [color={rgb, 255:red, 0; green, 0; blue, 0 }  ,draw opacity=1 ][line width=1.5]    (14.21,-6.37) .. controls (9.04,-2.99) and (4.3,-0.87) .. (0,0) .. controls (4.3,0.87) and (9.04,2.99) .. (14.21,6.37)   ;
\draw [color={rgb, 255:red, 128; green, 128; blue, 128 }  ,draw opacity=1 ][line width=1.5]    (274,89) .. controls (236.76,52.74) and (239.86,48.17) .. (272.47,20.71) ;
\draw [shift={(274.5,19)}, rotate = 499.95] [color={rgb, 255:red, 128; green, 128; blue, 128 }  ,draw opacity=1 ][line width=1.5]    (14.21,-6.37) .. controls (9.04,-2.99) and (4.3,-0.87) .. (0,0) .. controls (4.3,0.87) and (9.04,2.99) .. (14.21,6.37)   ;

\draw (23,9.4) node [anchor=north west][inner sep=0.75pt]    {$k$};
\draw (116,12.4) node [anchor=north west][inner sep=0.75pt]    {$l$};
\draw (22,77.4) node [anchor=north west][inner sep=0.75pt]    {$i$};
\draw (115,80.4) node [anchor=north west][inner sep=0.75pt]    {$j$};
\draw (87,43) node [anchor=north west][inner sep=0.75pt]  [color={rgb, 255:red, 0; green, 0; blue, 0 }  ,opacity=1 ]  {$x_c$};
\draw (34,40) node [anchor=north west][inner sep=0.75pt]  [color={rgb, 255:red, 128; green, 128; blue, 128 }  ,opacity=1 ]  {$-x_c$};
\draw (188,8.4) node [anchor=north west][inner sep=0.75pt]    {$k$};
\draw (278,11.4) node [anchor=north west][inner sep=0.75pt]    {$l$};
\draw (187,76.4) node [anchor=north west][inner sep=0.75pt]    {$i$};
\draw (277,79.4) node [anchor=north west][inner sep=0.75pt]    {$j$};
\draw (204,42) node [anchor=north west][inner sep=0.75pt]  [color={rgb, 255:red, 0; green, 0; blue, 0 }  ,opacity=1 ]  {$y_c$};
\draw (255,42) node [anchor=north west][inner sep=0.75pt]  [color={rgb, 255:red, 128; green, 128; blue, 128 }  ,opacity=1 ]  {$-y_c$};

\end{tikzpicture}
    \caption{Generators $x_c$ and $y_c$ of $R(D)$.}
    \label{fig:crossing}
\end{figure}

To each crossing $c$ of $D$, associate it an element $x_c \in R(D)$ by
\[
    x_c = X_l - X_i = -(X_k - X_j)
\]
where the indices $i, j, k, l$ are taken as in \Cref{fig:crossing}. Here we arbitrarily label the crossings as $c_1, \ldots, c_n$ and rewrite $x_{c_i}$ as $x_i$.

\begin{lemma} \label{lem:R-description}
    $R(D) = \QQ[x_1, \ldots, x_n].$
\end{lemma}

\begin{proof}
    Since we are working over $\QQ$, we may rewrite $R'(D)$ as 
    \[
        R'(D) = \QQ[X_1 - X_2,\ X_2 - X_3,\ \ldots,\ X_{2n-1} - X_{2n},\ \theta].
    \]
    Thus one sees that $R(D) = R'(D)/I(D)$ is generated by elements of the form $X_i - X_{i + 1}$ where $1 \leq i \leq 2n - 1$. It suffices to show that each $X_i - X_{i + 1}$ can be written as a linear combination of $x_j$'s. Since $D$ is a knot diagram, there is a unique oriented path $\gamma$ composed of edges in $G(D)$ that goes from $e_i$ to $e_{i + 1}$. Trace $\gamma$ from the start to the end, and every time $\gamma$ passes a crossing $c_j$ take a term $x_j$ or $-x_j$ as in \Cref{fig:crossing} according to how $\gamma$ passes $c_j$. It is obvious that the sum of these terms give $X_i - X_{i + 1}$ in $R(D)$.
\end{proof}

\begin{remark}
    For a link diagram $D$, a similar description of $R(D)$ can be obtained as follows. First to each crossing $c$ of $D$ associate it another element $y_c \in R(D)$ by
    \[
        y_c = X_k - X_i = -(X_l - X_j).
    \]
    Take a spanning tree $T$ of $G(D)$, and let $I$ the set of indices that correspond bijectively to the edges of $T$. Then
    \[
        R(D) = \QQ[x_i, y_j]_{i \in I, j \notin I}.
    \]
    After resolving all crossings which are not in $I$, we obtain an unknot diagram, and then the proof proceeds similarly. Although this procedure is not necessary for knot diagrams, it does contribute to reducing the computational cost when \textit{exclusion of variables} is applied (to be explained in \Cref{subsec:excl}). This is because the factor $X_k - X_i$ in the horizontal arrow of \Cref{fig:crossing} is replaced with a single term $y_c$. 
\end{remark}

\subsection{Slicing $C(D)$ by the $q$-grading} \label{subsec:double-cube-cpx}

Now we rewrite the definition of $\H(D)$ for the purpose of computation. First for each crossing $c_i$ of $D$, write the complex $C(D_{c_i})$ as 
\[ 
    \begin{tikzcd}
    R^{(i)}_{01} \arrow[r]           & R^{(i)}_{11}           \\
    R^{(i)}_{00} \arrow[r] \arrow[u] & R^{(i)}_{10} \arrow[u]
    \end{tikzcd} 
\]
To each pair of vertices $(h, v) \in \{0, 1\}^n \times \{0, 1\}^n$, associate a module $R_{h, v}$ by 
\[
    R_{h, v} = R^{(1)}_{h_1, v_1} \otimes \cdots \otimes R^{(n)}_{h_n, v_n}
\]
which is simply a copy of $R(D)$ with the grading shifted accordingly. Then $C(D)$ can be written as 
\[
    C(D) = \bigoplus_{h, v} R_{h, v}.
\]
For each $v \in \{0, 1\}^n$, the \textit{horizontal complex} $C_H(D, v)$ at $v$ is the subcomplex of $C(D)$ given by 
\[
    C_{H}(D, v) = \bigoplus_{h} R_{h, v}
\]
equipped with the horizontal differential $d_H$ that consists of horizontal arrows of the form $R_{h, v} \rightarrow R_{h', v}$ with $|h'| = |h| + 1$. Here we view each horizontal complex as an $n$-dimensional \textit{cube complex}, which is familiar in the construction of the Khovanov complex \cite{Kho00}. The \textit{horizontal homology} at $v$ is given by
\[
    H(D, v) = H(C_H(D, v), d_H).
\]
Next the \textit{vertical complex} is defined by 
\[
    C_V(D) = \bigoplus_{v} H(D, v),
\]
equipped with the vertical differential $(d_V)_*$ that consists of the induced maps of vertical arrows $H(D, v) \rightarrow H(D, v')$ with $|v'| = |v| + 1$. The vertical complex is again a cube complex with horizontal homologies on its vertices. Then  $\H(D)$ is given by the homology
\[
    \H(D) = H(C_V(D), d_V).
\]

Now obviously $C(D)$ is infinitely generated as a $\QQ$-vector space. In order to make possible the computation of $\H(D)$, we slice the complex $C(D)$ into homogeneous $q$-grading parts, each of which is finitely generated by monomials in $x_1, \ldots, x_n$. The homology of each of these slices can be computed directly by standard methods. From \Cref{prop:morton,prop:MFW-ineq} we know that $\H(D)$ is supported in a bounded region of $\ZZ^3$, hence the total computation is assured to be finite.

\subsection{Exclusion of variables}
\label{subsec:excl}

Although the computation is assured to be finite, the number of generators of the subcomplexes explodes as the level of the $q$-grading slice increases. This makes actual computations infeasible. In order to reduce the computational cost, we use the process called ``exclusion of variables" as described in \cite{Ras15}. 

First fix any $v \in \{0, 1\}^n$ and focus on the horizontal complex $C = C_H(D, v)$ at $v$. The differential $d = d_H$ is described by an $n$-tuple of polynomials $(f_1, \ldots, f_n)$. Take one such $f = f_i$. From the explicit description given in \Cref{fig:crossing}, the polynomial $f$ is monic and is either linear or quadratic with respect to some variable $x_k$. The complementary parts form an $n-1$ dimensional cube complex $C'$ with differential $d'$. The complex $C$ may be viewed as the mapping cone of the endomorphism 
\[
    C' \xrightarrow{f} C',
\]
and the differential can be written as
\[
    d(x_0, x_1) = (-d'x_0,\ fx_0 + d'x_1).
\]

Now let $R_1 = R/(f)$ be the quotient ring, and define $C'' = C' \otimes_R R_1$. In \cite[Lemma 3.8]{Ras15} it is proved that $C''$ is a deformation retract of $C$. Explicitly, let $\pi \colon R \rightarrow R_1$ be the quotient map, and $\iota \colon R_1 \rightarrow R$ be the map that sends any residue class $[g] \in R_1$ to the remainder of $g$ by $f$ with respect to the variable $x_k$. The chain homotopy equivalences $\phi \colon C \rightarrow C''$ and $\psi \colon C'' \rightarrow C$ are given by
\[
    \phi(x_0, x_1) = \pi(x_1),
\]
and
\[
    \psi(y) = \left(\frac{\iota(d''y) - d'\iota(y)}{f},\  \iota(y)\right).
\]

The effect of this process is that half of the cube is discarded, and moreover when $f$ is linear generators containing $x_k$ is dropped, when $f$ is quadratic generators containing $x_k^2$ is dropped. Furthermore, this process can be repeated for other directions as long as the target variable is algebraically independent in the quotient ring.

Having performed the reduction for all of the horizontal complexes, the vertices of the vertical complex can be computed from $H(C''_H(D, v), d''_H)$, and the vertical differential is given by induced map of the composite chain maps,
\[
    C''_H(D, v) \xrightarrow{\psi_v} C_H(D, v) \xrightarrow{d_V} C_H(D, v') \xrightarrow{\phi_{v'}} C''_H(D, v').
\] 

%
    \section{Computations} \label{sec:comptuations}

\begin{table}[t]
    \centering
    \begin{tabular}{r|rrrrrrr}
    $n \setminus l$	& $\leq 11$	& $12$ & $13$ & $14$ & $15$ & $16$ & $17$\\
    \hline
    $\leq 8$ & $35$ \\
    $9$ & $43$ & $5$ & & $1$ \\
    $10$ & $126$ & $31$ & $2$ & $6$ \\
    $11$ & $237$ & $135$ & $74$ & $81$ & $14$ & $9$ & $2$\\
    \end{tabular}
    \caption{Number of knots\ ($n$: crossing number, $l$: braid length)}
    \label{table:targets}
\end{table}

The implementation of the algorithm is available at \cite{kr-calc}, together with all of the computation results in both human-readble and machine-friendly formats. The program takes a braid representation of a knot as an input, and computes its reduced HOMFLY homology. For example, the following command runs the computation for the positive trefoil knot $3_1$, represented by the braid $b = \sigma_1\sigma_1\sigma_1$. 

\begin{verbatim}
$ yui-kr 3_1 -b [1,1,1]

Delta: 2
 j\i  -2  0  2 
 4    .   1  . 
 2    1   .  1 
\end{verbatim}

Here the output is expressed in the $\Delta$-sliced form, which was introduced in \cite{CG22}. The third grading $k$ can be read by $k = \Delta - i - j$, and each number represent the dimension of $\H^{i, j, k}$. This form has the advantage that the three-dimensional array can be expressed in a compact form, where also the symmetry (\Cref{prop:q-sym}) and KR-thinness becomes apparent. The above result reads
\[
    \H(3_1) = \QQ\{-2, 2, 2\} \oplus \QQ\{0, 4, -2\} \oplus \QQ\{2, 2, -2\}.
\]

The computational cost primarily depends on braid length $l$ of the input knot. \Cref{table:targets} shows the number of prime knots classified by the crossing number $n$ and the braid length $l$. Our program is capable of computing the results for knots with braid length $l \leq 13$ on an ordinary desktop computer. For higher $l$, we will need a high performance computer with a large amount of memory. Using the braid representatives provided in \textit{KnotInfo Database} as the input data, we run the program on a $64$-core $3$-TB RAM machine and succeeded to compute the homology for all prime knots in \Cref{table:targets}, except for one knot $11a_{280}\ (l = 17)$.

\begin{figure}[t]
    \centering
    \small
    \begin{verbatim}
Delta: -6     Delta: -4     Delta: -2           Delta: 0
 j\i  0        j\i  0        j\i  -2  0  2       j\i  -4  -2  0   2  4 
 2    ?        4    ?        4    .   ?  .       4    .   1   ?   1  . 
               2    ?        2    ?   ?  ?       2    1   ?   ?   ?  1 
               0    ?        0    .   ?  .       0    3   9   ?   9  3 
                                                 -2   2   8   11  8  2 
                                                 -4   .   3   5   3  . 
                                                 -6   .   .   1   .  . 


Delta: 2           Delta: 4           Delta: 6    Delta: 8    Delta: 10
 j\i  -2  0  2      j\i  -2  0  2     j\i  0       j\i  0      j\i  0 
 4    .   ?  .      4    .   ?  .     4    ?       4    ?      4    ? 
 2    ?   ?  ?      2    ?   ?  ?     2    ?       2    ?      2    ?
 0    .   ?  .      0    .   ?  .     0    ?   
    \end{verbatim}
    \vspace{-1.5em}
    \caption{Partial computation result for $11a_{280}$.}
    \label{fig:11a_280}
\end{figure}

For $11a_{280}$, we determined its homology by combining the partially computed result with known results. First, \Cref{fig:11a_280} shows the partial computation result for $11a_{280}$, where \verb|?| indicates the places where computations were skipped. Now recall the following result.

\begin{proposition}[{\cite[Theorem 2]{Ras15}}]
    For each $N > 0$, there is a spectral sequence $\{(E_p(N), d_p(N))\}_{p \geq 1}$ which starts at $E_1(N) = \H(K)$ and converges to the $sl(N)$-homology $H_{sl(N)}(K)$. Moreover the $p$-th differential $d_p(N)$ is homogeneous with degree $(2Np, -2p, 2 - 2p)$, and the bigrading $(t, q)$ on $H_{sl(N)}(K)$ and the triple grading $(i, j, k)$ on $E_\infty(N)$ correspond as $(t, q) = ((k - j)/2,\ i + Nj)$. 
\end{proposition}

In $\Delta$-grading, the $p$-th differential has degree $2(N - 2)p + 2$. Now for $K = 11a_{280}$, we see from \Cref{fig:11a_280} that when $N \geq 3$ all differentials $d_p(N)$ are trivial by degree reasons and hence $\H(K) \cong H_{sl(N)}(K)$. Lewark \cite{Lew13} has computed the $sl(3)$-homology for all prime knots with up to $12$ crossings, and the results show that $H_{sl(3)}(11a_{280})$ is homologically thin. Thus $\H(K)$ is thin, and all \verb|?| outside $\Delta = 0$ are determined to be trivial. The remaining five \verb|?| in $\Delta = 0$ can be filled from the HOMFLY polynomial of $K$.

Obviously this consideration can also be automated, and would generally improve the computation speed by determining trivial parts as early as possible. We leave it as a future work. 
    \section{Observations}\label{sec:observations}

Here we give more observations on the computational results. The following results show that the reduced HOMFLY homology is strictly stronger than the HOMFLY polynomial, as already known for example from \cite{Kawamuro}. 

\begin{proposition}
    $\{5_1, m(10_{132})\}$ and $\{11n_{79}, m(11n_{138})\}$ are pairs such that the two knots have identical HOMFLY polynomial but distinct HOMFLY homology. These are the only such pairs within prime knots with up to $11$ crossings. 
\end{proposition}

The first pair appears in \cite{BarNatan:2002} such that the two knots have identical Jones polynomial but distinct Khovanov homology. The second pair also satisfies this property. 

\begin{proposition}
    $9_{42}, 10_{125}, 11n_{24}, 11n_{82}$ are knots $K$ such that $K$ and $m(K)$ have identical HOMFLY polynomial but distinct reduced HOMFLY homology. These are the only knots with up to $11$ crossings satisfying this property. 
\end{proposition}

Contrarily, there are many pairs of knots that cannot be distinguished by HOMFLY homology.

\begin{proposition}
    Each set of knots in the following list have identical HOMFLY homology.
    \[
        \begin{array}{l}
            \{8_{8}, m(10_{129})\}, \{8_{16}, 10_{156}\}, \{10_{25}, 10_{56}\}, \{10_{40}, 10_{103}\}, \{10_{155}, 11n_{37}\}, \\ 
            \{11a_{1}, 11a_{149}\}, \{11a_{2}, 11a_{116}\}, \{11a_{11}, 11a_{167}\}, \{11a_{19}, 11a_{25}\}, \\
            \{11a_{24}, 11a_{26}, 11a_{315}\}, \{11a_{30}, 11a_{272}\}, \{11a_{33}, 11a_{82}\}, \{11a_{35}, 11a_{316}\}, \\
            \{11a_{41}, m(11a_{183})\}, \{11a_{44}, 11a_{47}\}, \{11a_{57}, 11a_{231}\}, \{11a_{71}, m(11a_{248})\}, \\
            \{11a_{76}, m(11a_{160}), m(11a_{289})\}, \{11a_{79}, 11a_{255}\}, \{11a_{81}, m(11a_{282})\}, \\
            \{11a_{104}, 11a_{168}\}, \{11a_{138}, m(11a_{285})\}, \{11a_{185}, m(11a_{265})\}, \\
            \{11a_{186}, 11a_{241}\}, \{11a_{192}, 11a_{299}\}, \{11a_{196}, m(11a_{216})\}, \{11a_{251}, 11a_{253}\}, \\
            \{11a_{252}, 11a_{254}\}, \{11n_{4}, 11n_{21}\}, \{11n_{10}, m(11n_{144})\}, \{11n_{11}, 11n_{112}\}, \\
            \{11n_{22}, m(11n_{127})\}, \{11n_{34}, 11n_{42}\}, \{11n_{35}, 11n_{43}\}, \{11n_{36}, 11n_{44}\}, \\
            \{11n_{39}, 11n_{45}\}, \{11n_{40}, 11n_{46}\}, \{11n_{41}, 11n_{47}\}, \{11n_{50}, 11n_{132}\}, \\
            \{11n_{56}, m(11n_{58})\}, \{11n_{71}, m(11n_{75})\}, \{11n_{73}, 11n_{74}\}, \{11n_{76}, m(11n_{78})\}, \\
            \{11n_{151}, 11n_{152}\}.
        \end{array}
    \]
\end{proposition}

In particular we see that the Conway knot ($11n_{34}$) and the Kinoshita--Terasaka knot ($11n_{42}$) have identical reduced HOMFLY homology. This was proved in \cite{MV08} using spectral sequence arguments. There are also knots whose chirality (i.e. $K \neq m(K)$) cannot be detected by HOMFLY homology.

\begin{proposition}
    The following four knots are chiral, but each of its HOMFLY homology is identical to that of its mirror:  $10_{48}$, $10_{71}$, $10_{91}$, $10_{104}$.
\end{proposition}

Finally we observe how $\H(D)$ may differ from $\H(K)$ when $D$ is a general knot diagram for $K$ (see \Cref{rem:generaldiag}). Below shows the computation result for a minimal crossing diagram for $5_2$.

\begin{verbatim}
Delta: 2
 j\i  -2  0  2 
 8    .   1  . 
 6    1   1  1 
 4    1   2  1 
 2    1   1  1 
 \end{verbatim}

However, the correct homology for $5_2$ is the following, obtained from a braid representation $b$ of length $6$.

\begin{verbatim}
Delta: 2
 j\i  -2  0  2 
 6    .   1  . 
 4    1   1  1 
 2    1   1  1 
\end{verbatim}

There are cases where a diagram $D$ of a knot $K$ gives the correct homology $\H(K) = \H(D)$ even when the number of crossings of $D$ are less than the braid length of $K$. For instance, a minimal crossing diagram of $11n_{39}$ represented by the following PD-code gives the correct result, while the braid length of $11n_{39}$ is $13$.
\begin{small}
\begin{verbatim}
  [[4,2,5,1],[8,4,9,3],[11,17,12,16],[12,5,13,6],[6,13,7,14],
   [17,22,18,1],[9,18,10,19],[21,10,22,11],[15,21,16,20],[19,15,20,14],
   [2,8,3,7]]
\end{verbatim}    
\end{small}
\begin{question} \label{question:refine-H}
    Is it possible to refine the construction of the chain complex $C(D)$ for arbitrary link diagrams so that it's homology gives the correct HOMFLY homology? 
\end{question}
    
    \printbibliography
    \vspace{1.5em}
    \addresses
    
    \newpage
    \newgeometry{margin=2.5cm}

\appendix
\section{Computation Results}\label{sec:results}

Below are the computation results of the reduced HOMFLY homology for KR-thick prime knots with up to $11$-crossings, expressed in the form of Poincar\'e series
\[
    \mathcal{P}(K) = \sum_{i,j,k} t^{(k-j)/2} q^i a^j \dim \H^{i,\,j,\,k}(K).
\]
The braid representatives of prime knots provided in \textit{KnotInfo Database} were used as the input data.

\vspace{1em}

\renewcommand*{\arraystretch}{1.5}
\small
\begin{longtable}{p{.05\textwidth}p{.9\textwidth}}
$8_{19}$ & $t^{-8}a^{10} + t^{-7}q^2a^8 + t^{-7}q^4a^8 + t^{-6}q^6a^6 + t^{-5}q^{-2}a^8 + t^{-5}a^8 + t^{-4}a^6 + t^{-4}q^2a^6 + t^{-3}q^{-4}a^8 + t^{-2}q^{-2}a^6 + q^{-6}a^6$ \\
$9_{42}$ & $t^{-4}q^2a^2 + t^{-3}q^4 + t^{-2}q^{-2}a^2 + 2t^{-1} + 1 + q^2a^{-2} + tq^{-4} + t^2q^{-2}a^{-2}$ \\
$9_{43}$ & $t^{-7}a^8 + t^{-6}q^2a^6 + t^{-6}q^4a^6 + t^{-5}q^2a^6 + t^{-5}q^6a^4 + t^{-4}q^{-2}a^6 + t^{-4}a^6 + t^{-4}q^4a^4 + t^{-3}q^{-2}a^6 + t^{-3}a^4 + 2t^{-3}q^2a^4 + t^{-2}q^{-4}a^6 + t^{-2}a^4 + t^{-2}q^4a^2 + 2t^{-1}q^{-2}a^4 + q^{-4}a^4 + a^2 + tq^{-6}a^4 + t^2q^{-4}a^2$ \\
$9_{47}$ & $t^{-6}a^6 + t^{-5}q^2a^4 + t^{-5}q^4a^4 + 2t^{-4}q^2a^4 + t^{-4}q^6a^2 + t^{-3}q^{-2}a^4 + 2t^{-3}a^4 + 2t^{-3}q^4a^2 + 2t^{-2}q^{-2}a^4 + t^{-2}a^2 + 3t^{-2}q^2a^2 + t^{-1}q^{-4}a^4 + 4t^{-1}a^2 + t^{-1}q^4 + 3q^{-2}a^2 + 2q^2 + 2tq^{-4}a^2 + t + t^2q^{-6}a^2 + 2t^2q^{-2} + t^3q^{-4}$ \\
$10_{124}$ & $t^{-10}q^2a^{12} + t^{-9}q^4a^{10} + t^{-9}q^6a^{10} + t^{-8}q^{-2}a^{12} + t^{-8}q^8a^8 + 2t^{-7}a^{10} + t^{-7}q^2a^{10} + t^{-6}q^2a^8 + t^{-6}q^4a^8 + t^{-5}q^{-4}a^{10} + t^{-5}q^{-2}a^{10} + t^{-4}q^{-2}a^8 + t^{-4}a^8 + t^{-3}q^{-6}a^{10} + t^{-2}q^{-4}a^8 + q^{-8}a^8$ \\
$10_{125}$ & $t^{-5}q^4a^2 + t^{-4}q^6 + t^{-3}a^2 + 2t^{-2}q^2 + t^{-1}q^{-4}a^2 + t^{-1}q^4a^{-2} + 2q^{-2} + 1 + ta^{-2} + t^2q^{-6} + t^3q^{-4}a^{-2}$ \\
$10_{128}$ & $t^{-10}a^{12} + t^{-9}q^2a^{10} + t^{-9}q^4a^{10} + t^{-8}a^{10} + t^{-8}q^6a^8 + t^{-7}q^{-2}a^{10} + t^{-7}a^{10} + t^{-7}q^2a^8 + t^{-7}q^4a^8 + t^{-6}a^8 + 2t^{-6}q^2a^8 + t^{-6}q^6a^6 + t^{-5}q^{-4}a^{10} + t^{-5}q^{-2}a^8 + t^{-5}a^8 + t^{-5}q^4a^6 + 2t^{-4}q^{-2}a^8 + t^{-4}a^6 + t^{-4}q^2a^6 + t^{-3}q^{-4}a^8 + t^{-3}a^6 + t^{-2}q^{-6}a^8 + t^{-2}q^{-2}a^6 + t^{-1}q^{-4}a^6 + q^{-6}a^6$ \\
$10_{132}$ & $q^2a^{-2} + tq^2a^{-2} + t^2q^{-2}a^{-2} + t^2q^4a^{-4} + t^3q^{-2}a^{-2} + t^3a^{-4} + 2t^4a^{-4} + t^5q^2a^{-6} + t^6q^{-4}a^{-4} + t^7q^{-2}a^{-6}$ \\
$10_{134}$ & $t^{-10}a^{12} + t^{-9}q^2a^{10} + t^{-9}q^4a^{10} + t^{-8}q^2a^{10} + t^{-8}q^6a^8 + t^{-7}q^{-2}a^{10} + t^{-7}a^{10} + 2t^{-7}q^4a^8 + t^{-6}q^{-2}a^{10} + t^{-6}a^8 + 2t^{-6}q^2a^8 + t^{-6}q^6a^6 + t^{-5}q^{-4}a^{10} + 3t^{-5}a^8 + t^{-5}q^4a^6 + 2t^{-4}q^{-2}a^8 + 2t^{-4}q^2a^6 + 2t^{-3}q^{-4}a^8 + t^{-3}a^6 + t^{-2}q^{-6}a^8 + 2t^{-2}q^{-2}a^6 + t^{-1}q^{-4}a^6 + q^{-6}a^6$ \\
$10_{136}$ & $t^{-4}q^2a^2 + t^{-3}a^2 + t^{-3}q^4 + t^{-2}q^{-2}a^2 + t^{-2}q^2 + 3t^{-1} + q^{-2} + 1 + 2q^2a^{-2} + tq^{-4} + ta^{-2} + 2t^2q^{-2}a^{-2} + t^3a^{-4}$ \\
$10_{138}$ & $t^{-6}a^6 + t^{-5}q^2a^4 + t^{-5}q^4a^4 + 2t^{-4}q^2a^4 + t^{-4}q^6a^2 + t^{-3}q^{-2}a^4 + 2t^{-3}a^4 + 2t^{-3}q^4a^2 + 2t^{-2}q^{-2}a^4 + t^{-2}a^2 + 4t^{-2}q^2a^2 + t^{-1}q^{-4}a^4 + 4t^{-1}a^2 + 2t^{-1}q^4 + 4q^{-2}a^2 + 2q^2 + 2tq^{-4}a^2 + 3t + t^2q^{-6}a^2 + 2t^2q^{-2} + t^2q^2a^{-2} + 2t^3q^{-4} + t^4q^{-2}a^{-2}$ \\
$10_{139}$ & $t^{-10}q^2a^{12} + t^{-9}a^{12} + t^{-9}q^4a^{10} + t^{-9}q^6a^{10} + t^{-8}q^{-2}a^{12} + t^{-8}q^2a^{10} + t^{-8}q^8a^8 + 2t^{-7}a^{10} + t^{-7}q^2a^{10} + t^{-6}q^{-2}a^{10} + t^{-6}q^2a^8 + t^{-6}q^4a^8 + t^{-5}q^{-4}a^{10} + t^{-5}q^{-2}a^{10} + t^{-5}a^8 + t^{-4}q^{-2}a^8 + t^{-4}a^8 + t^{-3}q^{-6}a^{10} + t^{-2}q^{-4}a^8 + q^{-8}a^8$ \\
$10_{142}$ & $t^{-10}a^{12} + t^{-9}q^2a^{10} + t^{-9}q^4a^{10} + t^{-8}q^6a^8 + t^{-7}q^{-2}a^{10} + t^{-7}a^{10} + t^{-7}q^4a^8 + t^{-6}a^8 + 2t^{-6}q^2a^8 + t^{-6}q^6a^6 + t^{-5}q^{-4}a^{10} + t^{-5}a^8 + t^{-5}q^4a^6 + 2t^{-4}q^{-2}a^8 + t^{-4}q^2a^6 + t^{-3}q^{-4}a^8 + t^{-3}a^6 + t^{-2}q^{-6}a^8 + t^{-2}q^{-2}a^6 + t^{-1}q^{-4}a^6 + q^{-6}a^6$ \\
$10_{145}$ & $q^4a^{-4} + t^2a^{-4} + t^3a^{-4} + t^3q^2a^{-6} + t^4q^{-4}a^{-4} + t^4q^2a^{-6} + t^5q^{-2}a^{-6} + t^5a^{-6} + t^6q^{-2}a^{-6} + t^6q^2a^{-8} + t^7a^{-8} + t^8q^{-2}a^{-8} + t^9a^{-10}$ \\
$10_{152}$ & $2t^{-10}q^2a^{12} + t^{-9}a^{12} + 2t^{-9}q^4a^{10} + t^{-9}q^6a^{10} + 2t^{-8}q^{-2}a^{12} + t^{-8}q^2a^{10} + t^{-8}q^8a^8 + 4t^{-7}a^{10} + t^{-7}q^2a^{10} + t^{-6}q^{-2}a^{10} + 2t^{-6}q^2a^8 + t^{-6}q^4a^8 + 2t^{-5}q^{-4}a^{10} + t^{-5}q^{-2}a^{10} + t^{-5}a^8 + 2t^{-4}q^{-2}a^8 + t^{-4}a^8 + t^{-3}q^{-6}a^{10} + t^{-2}q^{-4}a^8 + q^{-8}a^8$ \\
$10_{153}$ & $t^{-5}q^4a^2 + t^{-4}q^6 + t^{-3}a^2 + 2t^{-2}q^2 + t^{-1}q^{-4}a^2 + t^{-1}q^2 + t^{-1}q^4a^{-2} + 2q^{-2} + 2 + q^4a^{-2} + tq^{-2} + ta^{-2} + tq^2a^{-2} + t^2q^{-6} + 2t^2a^{-2} + t^3q^{-4}a^{-2} + t^3q^{-2}a^{-2} + t^3q^2a^{-4} + t^4q^{-4}a^{-2} + t^4a^{-4} + t^5q^{-2}a^{-4}$ \\
$10_{154}$ & $t^{-10}a^{12} + 2t^{-9}q^2a^{10} + 2t^{-8}a^{10} + t^{-8}q^4a^8 + 2t^{-7}q^{-2}a^{10} + 2t^{-7}q^2a^8 + t^{-7}q^4a^8 + 3t^{-6}a^8 + t^{-6}q^6a^6 + 2t^{-5}q^{-2}a^8 + t^{-5}a^8 + t^{-5}q^2a^6 + t^{-4}q^{-4}a^8 + 2t^{-4}a^6 + t^{-4}q^2a^6 + t^{-3}q^{-4}a^8 + t^{-3}q^{-2}a^6 + t^{-2}q^{-2}a^6 + q^{-6}a^6$ \\
$10_{160}$ & $t^{-7}a^8 + t^{-6}q^2a^6 + t^{-6}q^4a^6 + 2t^{-5}q^2a^6 + t^{-5}q^6a^4 + t^{-4}q^{-2}a^6 + t^{-4}a^6 + 2t^{-4}q^4a^4 + 2t^{-3}q^{-2}a^6 + t^{-3}a^4 + 2t^{-3}q^2a^4 + t^{-2}q^{-4}a^6 + 3t^{-2}a^4 + t^{-2}q^4a^2 + 2t^{-1}q^{-2}a^4 + t^{-1}q^2a^2 + 2q^{-4}a^4 + a^2 + tq^{-6}a^4 + tq^{-2}a^2 + t^2q^{-4}a^2$ \\
$10_{161}$ & $t^{-9}q^2a^{10} + t^{-8}a^{10} + t^{-8}q^4a^8 + t^{-7}q^{-2}a^{10} + t^{-7}q^2a^8 + t^{-7}q^4a^8 + 2t^{-6}a^8 + t^{-6}q^6a^6 + t^{-5}q^{-2}a^8 + t^{-5}a^8 + t^{-5}q^2a^6 + t^{-4}q^{-4}a^8 + t^{-4}a^6 + t^{-4}q^2a^6 + t^{-3}q^{-4}a^8 + t^{-3}q^{-2}a^6 + t^{-2}q^{-2}a^6 + q^{-6}a^6$ \\
$11a_{263}$ & $t^{-11}a^{14} + t^{-11}q^6a^{12} + t^{-10}q^2a^{12} + 3t^{-10}q^4a^{12} + t^{-10}q^8a^{10} + 3t^{-9}q^2a^{12} + 4t^{-9}q^6a^{10} + t^{-8}q^{-2}a^{12} + 5t^{-8}a^{12} + 4t^{-8}q^4a^{10} + t^{-8}q^8a^8 + 3t^{-7}q^{-2}a^{12} + t^{-7}a^{10} + 9t^{-7}q^2a^{10} + t^{-7}q^6a^8 + 3t^{-6}q^{-4}a^{12} + 6t^{-6}a^{10} + 4t^{-6}q^4a^8 + t^{-5}q^{-6}a^{12} + 9t^{-5}q^{-2}a^{10} + 3t^{-5}q^2a^8 + 4t^{-4}q^{-4}a^{10} + 6t^{-4}a^8 + 4t^{-3}q^{-6}a^{10} + 3t^{-3}q^{-2}a^8 + t^{-2}q^{-8}a^{10} + 4t^{-2}q^{-4}a^8 + t^{-1}q^{-6}a^8 + q^{-8}a^8$ \\
$11n_{2}$ & $t^{-9}a^{10} + t^{-8}q^2a^8 + t^{-8}q^4a^8 + 2t^{-7}q^2a^8 + t^{-7}q^6a^6 + t^{-6}q^{-2}a^8 + 2t^{-6}a^8 + 3t^{-6}q^4a^6 + 2t^{-5}q^{-2}a^8 + t^{-5}a^6 + 5t^{-5}q^2a^6 + t^{-5}q^6a^4 + t^{-4}q^{-4}a^8 + 6t^{-4}a^6 + 3t^{-4}q^4a^4 + 5t^{-3}q^{-2}a^6 + 5t^{-3}q^2a^4 + 3t^{-2}q^{-4}a^6 + 5t^{-2}a^4 + t^{-2}q^4a^2 + t^{-1}q^{-6}a^6 + 5t^{-1}q^{-2}a^4 + t^{-1}q^2a^2 + 3q^{-4}a^4 + 2a^2 + tq^{-6}a^4 + tq^{-2}a^2 + t^2q^{-4}a^2$ \\
$11n_{6}$ & $t^{-4}q^2a^2 + t^{-3}a^2 + t^{-3}q^4 + t^{-2}q^{-2}a^2 + t^{-2}q^2 + t^{-2}q^4 + 2t^{-1} + t^{-1}q^2 + t^{-1}q^6a^{-2} + q^{-2} + 2 + q^2a^{-2} + q^4a^{-2} + tq^{-4} + tq^{-2} + ta^{-2} + 3tq^2a^{-2} + t^2q^{-4} + t^2q^{-2}a^{-2} + 2t^2a^{-2} + 2t^2q^4a^{-4} + 3t^3q^{-2}a^{-2} + t^3q^2a^{-4} + t^4q^{-4}a^{-2} + 3t^4a^{-4} + t^5q^{-6}a^{-2} + t^5q^{-2}a^{-4} + t^5q^2a^{-6} + 2t^6q^{-4}a^{-4} + t^7q^{-2}a^{-6}$ \\
$11n_{9}$ & $2t^{-9}q^2a^{10} + 2t^{-8}a^{10} + 2t^{-8}q^4a^8 + t^{-8}q^6a^8 + 2t^{-7}q^{-2}a^{10} + 2t^{-7}q^2a^8 + t^{-7}q^4a^8 + t^{-7}q^8a^6 + 5t^{-6}a^8 + t^{-6}q^2a^8 + t^{-6}q^6a^6 + 2t^{-5}q^{-2}a^8 + t^{-5}a^8 + 3t^{-5}q^2a^6 + 2t^{-5}q^4a^6 + 2t^{-4}q^{-4}a^8 + t^{-4}q^{-2}a^8 + 2t^{-4}a^6 + t^{-4}q^2a^6 + t^{-4}q^6a^4 + t^{-3}q^{-4}a^8 + 3t^{-3}q^{-2}a^6 + 2t^{-3}a^6 + t^{-2}q^{-6}a^8 + t^{-2}q^{-2}a^6 + t^{-2}a^4 + t^{-2}q^2a^4 + 2t^{-1}q^{-4}a^6 + q^{-6}a^6 + q^{-2}a^4 + tq^{-8}a^6 + t^2q^{-6}a^4$ \\
$11n_{12}$ & $t^{-7}q^2a^6 + t^{-6}a^6 + t^{-6}q^4a^4 + t^{-5}q^{-2}a^6 + t^{-5}q^2a^4 + 3t^{-4}a^4 + t^{-3}q^{-2}a^4 + t^{-3}a^4 + 2t^{-3}q^2a^2 + t^{-2}q^{-4}a^4 + t^{-2}a^2 + t^{-2}q^2a^2 + 2t^{-1}q^{-2}a^2 + q^{-2}a^2 + 1$ \\
$11n_{13}$ & $t^{-9}q^2a^{10} + t^{-8}q^4a^8 + t^{-8}q^6a^8 + t^{-7}q^{-2}a^{10} + t^{-7}q^4a^8 + t^{-7}q^8a^6 + 2t^{-6}a^8 + t^{-6}q^2a^8 + t^{-6}q^6a^6 + t^{-5}a^8 + t^{-5}q^2a^6 + 2t^{-5}q^4a^6 + t^{-4}q^{-4}a^8 + t^{-4}q^{-2}a^8 + t^{-4}q^2a^6 + t^{-4}q^6a^4 + t^{-3}q^{-4}a^8 + t^{-3}q^{-2}a^6 + 2t^{-3}a^6 + t^{-2}q^{-6}a^8 + t^{-2}q^{-2}a^6 + t^{-2}q^2a^4 + 2t^{-1}q^{-4}a^6 + q^{-6}a^6 + q^{-2}a^4 + tq^{-8}a^6 + t^2q^{-6}a^4$ \\
$11n_{16}$ & $t^{-9}a^{10} + t^{-8}q^2a^8 + t^{-8}q^4a^8 + t^{-7}a^8 + t^{-7}q^2a^8 + t^{-7}q^6a^6 + t^{-6}q^{-2}a^8 + t^{-6}a^8 + t^{-6}q^2a^6 + 2t^{-6}q^4a^6 + t^{-5}q^{-2}a^8 + t^{-5}a^6 + 4t^{-5}q^2a^6 + t^{-5}q^6a^4 + t^{-4}q^{-4}a^8 + t^{-4}q^{-2}a^6 + 3t^{-4}a^6 + 3t^{-4}q^4a^4 + 4t^{-3}q^{-2}a^6 + t^{-3}a^4 + 3t^{-3}q^2a^4 + 2t^{-2}q^{-4}a^6 + 4t^{-2}a^4 + t^{-2}q^4a^2 + t^{-1}q^{-6}a^6 + 3t^{-1}q^{-2}a^4 + t^{-1}q^2a^2 + 3q^{-4}a^4 + a^2 + tq^{-6}a^4 + tq^{-2}a^2 + t^2q^{-4}a^2$ \\
$11n_{19}$ & $t^{-4}q^4 + t^{-3}q^6a^{-2} + t^{-2} + 2t^{-1}q^2a^{-2} + q^{-4} + q^2a^{-2} + q^4a^{-4} + 2tq^{-2}a^{-2} + ta^{-2} + t^2q^{-2}a^{-2} + t^2a^{-4} + t^2q^2a^{-4} + t^3q^{-6}a^{-2} + t^3a^{-4} + t^4q^{-4}a^{-4} + t^4q^{-2}a^{-4} + t^5a^{-6}$ \\
$11n_{20}$ & $t^{-4}q^2a^2 + t^{-3}a^2 + t^{-3}q^4 + t^{-2}q^{-2}a^2 + 2t^{-2}q^2 + 3t^{-1} + t^{-1}q^4a^{-2} + 2q^{-2} + 1 + 2q^2a^{-2} + tq^{-4} + 3ta^{-2} + 2t^2q^{-2}a^{-2} + t^2q^2a^{-4} + t^3q^{-4}a^{-2} + t^3a^{-4} + t^4q^{-2}a^{-4}$ \\
$11n_{23}$ & $t^{-8}q^2a^8 + t^{-7}a^8 + t^{-7}q^4a^6 + t^{-7}q^6a^6 + t^{-6}q^{-2}a^8 + t^{-6}q^2a^6 + t^{-6}q^4a^6 + t^{-6}q^8a^4 + 2t^{-5}a^6 + 3t^{-5}q^2a^6 + t^{-5}q^6a^4 + t^{-4}q^{-2}a^6 + t^{-4}a^6 + t^{-4}q^2a^4 + 4t^{-4}q^4a^4 + t^{-3}q^{-4}a^6 + 3t^{-3}q^{-2}a^6 + t^{-3}a^4 + 2t^{-3}q^2a^4 + t^{-3}q^6a^2 + t^{-2}q^{-4}a^6 + t^{-2}q^{-2}a^4 + 5t^{-2}a^4 + t^{-2}q^4a^2 + t^{-1}q^{-6}a^6 + 2t^{-1}q^{-2}a^4 + 2t^{-1}q^2a^2 + 4q^{-4}a^4 + a^2 + tq^{-6}a^4 + 2tq^{-2}a^2 + t^2q^{-8}a^4 + t^2q^{-4}a^2 + t^3q^{-6}a^2$ \\
$11n_{24}$ & $t^{-3}q^4a^2 + t^{-2}q^2a^2 + t^{-2}q^6 + 2t^{-1}a^2 + t^{-1}q^4 + q^{-2}a^2 + 1 + 3q^2 + tq^{-4}a^2 + 2t + tq^4a^{-2} + 3t^2q^{-2} + t^2q^2a^{-2} + t^3q^{-4} + 2t^3a^{-2} + t^4q^{-6} + t^4q^{-2}a^{-2} + t^5q^{-4}a^{-2}$ \\
$11n_{27}$ & $t^{-9}q^2a^{10} + t^{-8}a^{10} + t^{-8}q^4a^8 + t^{-8}q^6a^8 + t^{-7}q^{-2}a^{10} + t^{-7}q^2a^8 + t^{-7}q^4a^8 + t^{-7}q^8a^6 + 2t^{-6}a^8 + 2t^{-6}q^2a^8 + t^{-6}q^6a^6 + t^{-5}q^{-2}a^8 + t^{-5}a^8 + t^{-5}q^2a^6 + 3t^{-5}q^4a^6 + t^{-4}q^{-4}a^8 + 2t^{-4}q^{-2}a^8 + t^{-4}a^6 + t^{-4}q^2a^6 + t^{-4}q^6a^4 + t^{-3}q^{-4}a^8 + t^{-3}q^{-2}a^6 + 4t^{-3}a^6 + t^{-2}q^{-6}a^8 + t^{-2}q^{-2}a^6 + 2t^{-2}q^2a^4 + 3t^{-1}q^{-4}a^6 + q^{-6}a^6 + 2q^{-2}a^4 + tq^{-8}a^6 + t^2q^{-6}a^4$ \\
$11n_{30}$ & $t^{-9}a^{10} + t^{-8}q^2a^8 + t^{-8}q^4a^8 + t^{-7}q^2a^8 + t^{-7}q^6a^6 + t^{-6}q^{-2}a^8 + t^{-6}a^8 + 2t^{-6}q^4a^6 + t^{-5}q^{-2}a^8 + t^{-5}a^6 + 3t^{-5}q^2a^6 + t^{-5}q^6a^4 + t^{-4}q^{-4}a^8 + 3t^{-4}a^6 + 2t^{-4}q^4a^4 + 3t^{-3}q^{-2}a^6 + 3t^{-3}q^2a^4 + 2t^{-2}q^{-4}a^6 + 2t^{-2}a^4 + t^{-2}q^4a^2 + t^{-1}q^{-6}a^6 + 3t^{-1}q^{-2}a^4 + 2q^{-4}a^4 + a^2 + tq^{-6}a^4 + t^2q^{-4}a^2$ \\
$11n_{31}$ & $t^{-9}a^{10} + 2t^{-8}q^2a^8 + 2t^{-7}a^8 + t^{-7}q^4a^6 + 2t^{-6}q^{-2}a^8 + 2t^{-6}q^2a^6 + t^{-6}q^4a^6 + 3t^{-5}a^6 + t^{-5}q^2a^6 + t^{-5}q^6a^4 + 2t^{-4}q^{-2}a^6 + t^{-4}a^6 + t^{-4}q^2a^4 + t^{-4}q^4a^4 + t^{-3}q^{-4}a^6 + t^{-3}q^{-2}a^6 + 2t^{-3}a^4 + 2t^{-3}q^2a^4 + t^{-2}q^{-4}a^6 + t^{-2}q^{-2}a^4 + t^{-2}a^4 + t^{-2}q^4a^2 + 2t^{-1}q^{-2}a^4 + q^{-4}a^4 + a^2 + tq^{-6}a^4 + t^2q^{-4}a^2$ \\
$11n_{34}$ & $t^{-5}q^4a^2 + t^{-4}q^2a^2 + t^{-4}q^6 + 2t^{-3}a^2 + t^{-3}q^4 + t^{-2}q^{-2}a^2 + 3t^{-2}q^2 + t^{-2}q^4 + t^{-1}q^{-4}a^2 + 2t^{-1} + t^{-1}q^2 + t^{-1}q^4a^{-2} + t^{-1}q^6a^{-2} + 3q^{-2} + 3 + q^2a^{-2} + q^4a^{-2} + tq^{-4} + tq^{-2} + 2ta^{-2} + 3tq^2a^{-2} + t^2q^{-6} + t^2q^{-4} + t^2q^{-2}a^{-2} + 2t^2a^{-2} + t^2q^4a^{-4} + t^3q^{-4}a^{-2} + 3t^3q^{-2}a^{-2} + t^3q^2a^{-4} + t^4q^{-4}a^{-2} + 2t^4a^{-4} + t^5q^{-6}a^{-2} + t^5q^{-2}a^{-4} + t^6q^{-4}a^{-4}$ \\
$11n_{36}$ & $t^{-7}q^2a^6 + t^{-6}q^4a^4 + t^{-6}q^6a^4 + t^{-5}q^{-2}a^6 + 2t^{-5}q^4a^4 + t^{-5}q^8a^2 + 2t^{-4}a^4 + 4t^{-4}q^2a^4 + 2t^{-4}q^6a^2 + 5t^{-3}a^4 + t^{-3}q^2a^2 + 5t^{-3}q^4a^2 + t^{-2}q^{-4}a^4 + 4t^{-2}q^{-2}a^4 + 7t^{-2}q^2a^2 + t^{-2}q^6 + 2t^{-1}q^{-4}a^4 + t^{-1}q^{-2}a^2 + 8t^{-1}a^2 + 2t^{-1}q^4 + q^{-6}a^4 + 7q^{-2}a^2 + 4q^2 + 5tq^{-4}a^2 + 4t + 2t^2q^{-6}a^2 + 4t^2q^{-2} + t^3q^{-8}a^2 + 2t^3q^{-4} + t^4q^{-6}$ \\
$11n_{38}$ & $t^{-2}a^2 + t^{-1}q^2 + 1 + q^2 + tq^{-2} + tq^4a^{-2} + t^2q^{-2} + t^2a^{-2} + 2t^3a^{-2} + t^4q^2a^{-4} + t^5q^{-4}a^{-2} + t^6q^{-2}a^{-4}$ \\
$11n_{39}$ & $t^{-6}q^4a^4 + 2t^{-5}q^2a^4 + t^{-5}q^6a^2 + 3t^{-4}a^4 + 2t^{-4}q^4a^2 + 2t^{-3}q^{-2}a^4 + 4t^{-3}q^2a^2 + t^{-3}q^4a^2 + t^{-2}q^{-4}a^4 + 5t^{-2}a^2 + t^{-2}q^4 + t^{-2}q^6 + 4t^{-1}q^{-2}a^2 + 2t^{-1}a^2 + 3t^{-1}q^2 + 2q^{-4}a^2 + 4 + 3q^2 + tq^{-6}a^2 + tq^{-4}a^2 + 3tq^{-2} + tq^4a^{-2} + t^2q^{-4} + 3t^2q^{-2} + t^2a^{-2} + 2t^3a^{-2} + t^4q^{-6} + t^5q^{-4}a^{-2}$ \\
$11n_{41}$ & $t^{-8}q^2a^8 + t^{-7}a^8 + t^{-7}q^4a^6 + t^{-7}q^6a^6 + t^{-6}q^{-2}a^8 + t^{-6}q^2a^6 + 2t^{-6}q^4a^6 + t^{-6}q^8a^4 + 2t^{-5}a^6 + 4t^{-5}q^2a^6 + 2t^{-5}q^6a^4 + t^{-4}q^{-2}a^6 + 3t^{-4}a^6 + t^{-4}q^2a^4 + 5t^{-4}q^4a^4 + t^{-3}q^{-4}a^6 + 4t^{-3}q^{-2}a^6 + t^{-3}a^4 + 5t^{-3}q^2a^4 + t^{-3}q^6a^2 + 2t^{-2}q^{-4}a^6 + t^{-2}q^{-2}a^4 + 7t^{-2}a^4 + 2t^{-2}q^4a^2 + t^{-1}q^{-6}a^6 + 5t^{-1}q^{-2}a^4 + 3t^{-1}q^2a^2 + 5q^{-4}a^4 + 3a^2 + 2tq^{-6}a^4 + 3tq^{-2}a^2 + t^2q^{-8}a^4 + 2t^2q^{-4}a^2 + t^3q^{-6}a^2$ \\
$11n_{42}$ & $t^{-5}q^4a^2 + t^{-4}q^2a^2 + t^{-4}q^6 + 2t^{-3}a^2 + t^{-3}q^4 + t^{-2}q^{-2}a^2 + 3t^{-2}q^2 + t^{-2}q^4 + t^{-1}q^{-4}a^2 + 2t^{-1} + t^{-1}q^2 + t^{-1}q^4a^{-2} + t^{-1}q^6a^{-2} + 3q^{-2} + 3 + q^2a^{-2} + q^4a^{-2} + tq^{-4} + tq^{-2} + 2ta^{-2} + 3tq^2a^{-2} + t^2q^{-6} + t^2q^{-4} + t^2q^{-2}a^{-2} + 2t^2a^{-2} + t^2q^4a^{-4} + t^3q^{-4}a^{-2} + 3t^3q^{-2}a^{-2} + t^3q^2a^{-4} + t^4q^{-4}a^{-2} + 2t^4a^{-4} + t^5q^{-6}a^{-2} + t^5q^{-2}a^{-4} + t^6q^{-4}a^{-4}$ \\
$11n_{44}$ & $t^{-7}q^2a^6 + t^{-6}q^4a^4 + t^{-6}q^6a^4 + t^{-5}q^{-2}a^6 + 2t^{-5}q^4a^4 + t^{-5}q^8a^2 + 2t^{-4}a^4 + 4t^{-4}q^2a^4 + 2t^{-4}q^6a^2 + 5t^{-3}a^4 + t^{-3}q^2a^2 + 5t^{-3}q^4a^2 + t^{-2}q^{-4}a^4 + 4t^{-2}q^{-2}a^4 + 7t^{-2}q^2a^2 + t^{-2}q^6 + 2t^{-1}q^{-4}a^4 + t^{-1}q^{-2}a^2 + 8t^{-1}a^2 + 2t^{-1}q^4 + q^{-6}a^4 + 7q^{-2}a^2 + 4q^2 + 5tq^{-4}a^2 + 4t + 2t^2q^{-6}a^2 + 4t^2q^{-2} + t^3q^{-8}a^2 + 2t^3q^{-4} + t^4q^{-6}$ \\
$11n_{45}$ & $t^{-6}q^4a^4 + 2t^{-5}q^2a^4 + t^{-5}q^6a^2 + 3t^{-4}a^4 + 2t^{-4}q^4a^2 + 2t^{-3}q^{-2}a^4 + 4t^{-3}q^2a^2 + t^{-3}q^4a^2 + t^{-2}q^{-4}a^4 + 5t^{-2}a^2 + t^{-2}q^4 + t^{-2}q^6 + 4t^{-1}q^{-2}a^2 + 2t^{-1}a^2 + 3t^{-1}q^2 + 2q^{-4}a^2 + 4 + 3q^2 + tq^{-6}a^2 + tq^{-4}a^2 + 3tq^{-2} + tq^4a^{-2} + t^2q^{-4} + 3t^2q^{-2} + t^2a^{-2} + 2t^3a^{-2} + t^4q^{-6} + t^5q^{-4}a^{-2}$ \\
$11n_{47}$ & $t^{-8}q^2a^8 + t^{-7}a^8 + t^{-7}q^4a^6 + t^{-7}q^6a^6 + t^{-6}q^{-2}a^8 + t^{-6}q^2a^6 + 2t^{-6}q^4a^6 + t^{-6}q^8a^4 + 2t^{-5}a^6 + 4t^{-5}q^2a^6 + 2t^{-5}q^6a^4 + t^{-4}q^{-2}a^6 + 3t^{-4}a^6 + t^{-4}q^2a^4 + 5t^{-4}q^4a^4 + t^{-3}q^{-4}a^6 + 4t^{-3}q^{-2}a^6 + t^{-3}a^4 + 5t^{-3}q^2a^4 + t^{-3}q^6a^2 + 2t^{-2}q^{-4}a^6 + t^{-2}q^{-2}a^4 + 7t^{-2}a^4 + 2t^{-2}q^4a^2 + t^{-1}q^{-6}a^6 + 5t^{-1}q^{-2}a^4 + 3t^{-1}q^2a^2 + 5q^{-4}a^4 + 3a^2 + 2tq^{-6}a^4 + 3tq^{-2}a^2 + t^2q^{-8}a^4 + 2t^2q^{-4}a^2 + t^3q^{-6}a^2$ \\
$11n_{49}$ & $t^{-4}a^4 + t^{-3}q^2a^2 + t^{-2}a^2 + t^{-1}q^{-2}a^2 + t^{-1}q^2 + 2 + q^2 + tq^{-2} + tq^4a^{-2} + t^2q^{-2} + t^2a^{-2} + 2t^3a^{-2} + t^4q^2a^{-4} + t^5q^{-4}a^{-2} + t^6q^{-2}a^{-4}$ \\
$11n_{57}$ & $t^{-9}q^2a^{10} + t^{-8}a^{10} + t^{-8}q^4a^8 + t^{-8}q^6a^8 + t^{-7}q^{-2}a^{10} + t^{-7}q^2a^8 + t^{-7}q^4a^8 + t^{-7}q^8a^6 + 3t^{-6}a^8 + t^{-6}q^2a^8 + t^{-6}q^6a^6 + t^{-5}q^{-2}a^8 + t^{-5}a^8 + 2t^{-5}q^2a^6 + 2t^{-5}q^4a^6 + t^{-4}q^{-4}a^8 + t^{-4}q^{-2}a^8 + t^{-4}a^6 + t^{-4}q^2a^6 + t^{-4}q^6a^4 + t^{-3}q^{-4}a^8 + 2t^{-3}q^{-2}a^6 + 2t^{-3}a^6 + t^{-2}q^{-6}a^8 + t^{-2}q^{-2}a^6 + t^{-2}a^4 + t^{-2}q^2a^4 + 2t^{-1}q^{-4}a^6 + q^{-6}a^6 + q^{-2}a^4 + tq^{-8}a^6 + t^2q^{-6}a^4$ \\
$11n_{60}$ & $t^{-7}q^2a^6 + t^{-6}q^4a^4 + t^{-6}q^6a^4 + t^{-5}q^{-2}a^6 + t^{-5}q^4a^4 + t^{-5}q^8a^2 + 2t^{-4}a^4 + 2t^{-4}q^2a^4 + t^{-4}q^6a^2 + 2t^{-3}a^4 + t^{-3}q^2a^2 + 3t^{-3}q^4a^2 + t^{-2}q^{-4}a^4 + 2t^{-2}q^{-2}a^4 + 3t^{-2}q^2a^2 + t^{-2}q^6 + t^{-1}q^{-4}a^4 + t^{-1}q^{-2}a^2 + 4t^{-1}a^2 + t^{-1}q^4 + q^{-6}a^4 + 3q^{-2}a^2 + 2q^2 + 3tq^{-4}a^2 + t + t^2q^{-6}a^2 + 2t^2q^{-2} + t^3q^{-8}a^2 + t^3q^{-4} + t^4q^{-6}$ \\
$11n_{61}$ & $t^{-8}q^2a^8 + t^{-7}a^8 + t^{-7}q^4a^6 + t^{-7}q^6a^6 + t^{-6}q^{-2}a^8 + t^{-6}q^2a^6 + t^{-6}q^4a^6 + t^{-6}q^8a^4 + 3t^{-5}a^6 + 2t^{-5}q^2a^6 + t^{-5}q^6a^4 + t^{-4}q^{-2}a^6 + t^{-4}a^6 + 2t^{-4}q^2a^4 + 3t^{-4}q^4a^4 + t^{-3}q^{-4}a^6 + 2t^{-3}q^{-2}a^6 + t^{-3}a^4 + 2t^{-3}q^2a^4 + t^{-3}q^6a^2 + t^{-2}q^{-4}a^6 + 2t^{-2}q^{-2}a^4 + 3t^{-2}a^4 + t^{-2}q^4a^2 + t^{-1}q^{-6}a^6 + 2t^{-1}q^{-2}a^4 + t^{-1}a^2 + t^{-1}q^2a^2 + 3q^{-4}a^4 + a^2 + tq^{-6}a^4 + tq^{-2}a^2 + t^2q^{-8}a^4 + t^2q^{-4}a^2 + t^3q^{-6}a^2$ \\
$11n_{64}$ & $t^{-7}a^8 + t^{-6}q^2a^6 + t^{-6}q^4a^6 + t^{-5}a^6 + t^{-5}q^2a^6 + t^{-5}q^6a^4 + t^{-4}q^{-2}a^6 + t^{-4}a^6 + t^{-4}q^2a^4 + 2t^{-4}q^4a^4 + t^{-3}q^{-2}a^6 + t^{-3}a^4 + 2t^{-3}q^2a^4 + t^{-3}q^6a^2 + t^{-2}q^{-4}a^6 + t^{-2}q^{-2}a^4 + 2t^{-2}a^4 + t^{-2}q^4a^2 + 2t^{-1}q^{-2}a^4 + t^{-1}a^2 + 2t^{-1}q^2a^2 + 2q^{-4}a^4 + a^2 + q^4 + tq^{-6}a^4 + 2tq^{-2}a^2 + t^2q^{-4}a^2 + t^2 + t^3q^{-6}a^2 + t^4q^{-4}$ \\
$11n_{67}$ & $t^{-7}q^2a^6 + t^{-6}a^6 + t^{-6}q^4a^4 + t^{-5}q^{-2}a^6 + t^{-5}q^2a^4 + 3t^{-4}a^4 + t^{-3}q^{-2}a^4 + 2t^{-3}q^2a^2 + t^{-2}q^{-4}a^4 + t^{-2}a^2 + t^{-2}q^2a^2 + 2t^{-1}q^{-2}a^2 + t^{-1}q^4 + q^{-2}a^2 + 2 + 2t + t^2q^2a^{-2} + t^3q^{-4} + t^4q^{-2}a^{-2}$ \\
$11n_{69}$ & $t^{-9}a^{10} + t^{-8}q^2a^8 + t^{-8}q^4a^8 + t^{-7}q^2a^8 + t^{-7}q^6a^6 + t^{-6}q^{-2}a^8 + t^{-6}a^8 + 2t^{-6}q^4a^6 + t^{-5}q^{-2}a^8 + t^{-5}a^6 + 4t^{-5}q^2a^6 + t^{-5}q^6a^4 + t^{-4}q^{-4}a^8 + 4t^{-4}a^6 + 3t^{-4}q^4a^4 + 4t^{-3}q^{-2}a^6 + 4t^{-3}q^2a^4 + 2t^{-2}q^{-4}a^6 + 4t^{-2}a^4 + t^{-2}q^4a^2 + t^{-1}q^{-6}a^6 + 4t^{-1}q^{-2}a^4 + t^{-1}q^2a^2 + 3q^{-4}a^4 + 2a^2 + tq^{-6}a^4 + tq^{-2}a^2 + t^2q^{-4}a^2$ \\
$11n_{70}$ & $t^{-4}q^4a^4 + t^{-3}a^4 + t^{-3}q^6a^2 + 2t^{-2}a^4 + t^{-2}q^2a^2 + 3t^{-1}q^2a^2 + q^{-4}a^4 + q^{-2}a^2 + q^4 + 3tq^{-2}a^2 + 2t^2 + t^3q^{-6}a^2 + t^4q^{-4}$ \\
$11n_{73}$ & $t^{-6}q^4a^4 + t^{-5}q^2a^4 + t^{-5}q^6a^2 + 2t^{-4}a^4 + t^{-4}q^4a^2 + t^{-3}q^{-2}a^4 + 3t^{-3}q^2a^2 + t^{-3}q^4a^2 + t^{-2}q^{-4}a^4 + 2t^{-2}a^2 + t^{-2}q^4 + t^{-2}q^6 + 3t^{-1}q^{-2}a^2 + 2t^{-1}a^2 + t^{-1}q^2 + q^{-4}a^2 + 3 + 3q^2 + tq^{-6}a^2 + tq^{-4}a^2 + tq^{-2} + tq^4a^{-2} + t^2q^{-4} + 3t^2q^{-2} + 2t^3a^{-2} + t^4q^{-6} + t^5q^{-4}a^{-2}$ \\
$11n_{74}$ & $t^{-6}q^4a^4 + t^{-5}q^2a^4 + t^{-5}q^6a^2 + 2t^{-4}a^4 + t^{-4}q^4a^2 + t^{-3}q^{-2}a^4 + 3t^{-3}q^2a^2 + t^{-3}q^4a^2 + t^{-2}q^{-4}a^4 + 2t^{-2}a^2 + t^{-2}q^4 + t^{-2}q^6 + 3t^{-1}q^{-2}a^2 + 2t^{-1}a^2 + t^{-1}q^2 + q^{-4}a^2 + 3 + 3q^2 + tq^{-6}a^2 + tq^{-4}a^2 + tq^{-2} + tq^4a^{-2} + t^2q^{-4} + 3t^2q^{-2} + 2t^3a^{-2} + t^4q^{-6} + t^5q^{-4}a^{-2}$ \\
$11n_{76}$ & $t^{-3}q^6a^{-2} + t^{-2}q^4a^{-2} + t^{-2}q^8a^{-4} + 3t^{-1}q^2a^{-2} + t^{-1}q^6a^{-4} + 2a^{-2} + 5q^4a^{-4} + 3tq^{-2}a^{-2} + 3tq^2a^{-4} + tq^6a^{-6} + t^2q^{-4}a^{-2} + 7t^2a^{-4} + t^2q^2a^{-4} + t^2q^4a^{-6} + t^3q^{-6}a^{-2} + 3t^3q^{-2}a^{-4} + 4t^3q^2a^{-6} + t^3q^4a^{-6} + 5t^4q^{-4}a^{-4} + t^4q^{-2}a^{-4} + 2t^4a^{-6} + t^5q^{-6}a^{-4} + 4t^5q^{-2}a^{-6} + 2t^5a^{-6} + t^6q^{-8}a^{-4} + t^6q^{-4}a^{-6} + t^6q^2a^{-8} + t^7q^{-6}a^{-6} + t^7q^{-4}a^{-6} + t^8q^{-2}a^{-8}$ \\
$11n_{77}$ & $t^{-11}a^{14} + 4t^{-10}q^2a^{12} + 2t^{-9}a^{12} + 3t^{-9}q^4a^{10} + t^{-9}q^6a^{10} + 4t^{-8}q^{-2}a^{12} + 2t^{-8}q^2a^{10} + t^{-8}q^8a^8 + 7t^{-7}a^{10} + t^{-7}q^2a^{10} + 2t^{-6}q^{-2}a^{10} + 3t^{-6}q^2a^8 + t^{-6}q^4a^8 + 3t^{-5}q^{-4}a^{10} + t^{-5}q^{-2}a^{10} + 2t^{-5}a^8 + 3t^{-4}q^{-2}a^8 + t^{-4}a^8 + t^{-3}q^{-6}a^{10} + t^{-2}q^{-4}a^8 + q^{-8}a^8$ \\
$11n_{78}$ & $t^{-8}q^2a^8 + t^{-7}q^4a^6 + t^{-7}q^6a^6 + t^{-6}q^{-2}a^8 + t^{-6}q^4a^6 + t^{-6}q^8a^4 + 2t^{-5}a^6 + 4t^{-5}q^2a^6 + t^{-5}q^6a^4 + 2t^{-4}a^6 + t^{-4}q^2a^4 + 5t^{-4}q^4a^4 + t^{-3}q^{-4}a^6 + 4t^{-3}q^{-2}a^6 + 3t^{-3}q^2a^4 + t^{-3}q^6a^2 + t^{-2}q^{-4}a^6 + t^{-2}q^{-2}a^4 + 7t^{-2}a^4 + t^{-2}q^4a^2 + t^{-1}q^{-6}a^6 + 3t^{-1}q^{-2}a^4 + 3t^{-1}q^2a^2 + 5q^{-4}a^4 + 2a^2 + tq^{-6}a^4 + 3tq^{-2}a^2 + t^2q^{-8}a^4 + t^2q^{-4}a^2 + t^3q^{-6}a^2$ \\
$11n_{79}$ & $t^{-4}q^2a^4 + t^{-3}q^4a^2 + t^{-2}q^{-2}a^4 + t^{-2}q^2a^2 + 2t^{-1}a^2 + t^{-1}q^4 + q^{-2}a^2 + 1 + q^2 + tq^{-4}a^2 + 2t + t^2q^{-2} + t^2q^2a^{-2} + t^3q^{-4} + t^4q^{-2}a^{-2}$ \\
$11n_{80}$ & $t^{-4}q^4 + t^{-3}q^6a^{-2} + t^{-2} + 2t^{-1}q^2a^{-2} + q^{-4} + 2q^2a^{-2} + q^4a^{-4} + 2tq^{-2}a^{-2} + 3ta^{-2} + tq^4a^{-4} + 2t^2q^{-2}a^{-2} + t^2a^{-4} + 3t^2q^2a^{-4} + t^3q^{-6}a^{-2} + 4t^3a^{-4} + t^4q^{-4}a^{-4} + 3t^4q^{-2}a^{-4} + 2t^4q^2a^{-6} + t^5q^{-4}a^{-4} + 3t^5a^{-6} + 2t^6q^{-2}a^{-6} + t^7a^{-8}$ \\
$11n_{81}$ & $t^{-9}q^2a^{10} + t^{-8}a^{10} + t^{-8}q^4a^8 + t^{-8}q^6a^8 + t^{-7}q^{-2}a^{10} + t^{-7}q^2a^8 + t^{-7}q^4a^8 + t^{-7}q^8a^6 + 2t^{-6}a^8 + 3t^{-6}q^2a^8 + t^{-6}q^6a^6 + t^{-5}q^{-2}a^8 + t^{-5}a^8 + t^{-5}q^2a^6 + 4t^{-5}q^4a^6 + t^{-4}q^{-4}a^8 + 3t^{-4}q^{-2}a^8 + t^{-4}a^6 + t^{-4}q^2a^6 + t^{-4}q^6a^4 + t^{-3}q^{-4}a^8 + t^{-3}q^{-2}a^6 + 6t^{-3}a^6 + t^{-2}q^{-6}a^8 + t^{-2}q^{-2}a^6 + 3t^{-2}q^2a^4 + 4t^{-1}q^{-4}a^6 + q^{-6}a^6 + 3q^{-2}a^4 + tq^{-8}a^6 + t^2q^{-6}a^4$ \\
$11n_{82}$ & $t^{-3}q^4a^2 + t^{-2}q^2a^2 + t^{-2}q^6 + t^{-1}a^2 + t^{-1}q^4 + q^{-2}a^2 + 1 + 2q^2 + tq^{-4}a^2 + 2t + tq^4a^{-2} + 2t^2q^{-2} + t^2q^2a^{-2} + t^3q^{-4} + t^3a^{-2} + t^4q^{-6} + t^4q^{-2}a^{-2} + t^5q^{-4}a^{-2}$ \\
$11n_{88}$ & $t^{-9}q^2a^{10} + t^{-8}a^{10} + t^{-8}q^4a^8 + t^{-8}q^6a^8 + t^{-7}q^{-2}a^{10} + t^{-7}q^2a^8 + t^{-7}q^4a^8 + t^{-7}q^8a^6 + 2t^{-6}a^8 + t^{-6}q^2a^8 + t^{-6}q^6a^6 + t^{-5}q^{-2}a^8 + t^{-5}a^8 + t^{-5}q^2a^6 + 2t^{-5}q^4a^6 + t^{-4}q^{-4}a^8 + t^{-4}q^{-2}a^8 + t^{-4}a^6 + t^{-4}q^2a^6 + t^{-4}q^6a^4 + t^{-3}q^{-4}a^8 + t^{-3}q^{-2}a^6 + 2t^{-3}a^6 + t^{-2}q^{-6}a^8 + t^{-2}q^{-2}a^6 + t^{-2}q^2a^4 + 2t^{-1}q^{-4}a^6 + q^{-6}a^6 + q^{-2}a^4 + tq^{-8}a^6 + t^2q^{-6}a^4$ \\
$11n_{90}$ & $t^{-9}a^{10} + t^{-8}q^2a^8 + t^{-8}q^4a^8 + t^{-7}q^2a^8 + t^{-7}q^6a^6 + t^{-6}q^{-2}a^8 + t^{-6}a^8 + 2t^{-6}q^4a^6 + t^{-5}q^{-2}a^8 + t^{-5}a^6 + 4t^{-5}q^2a^6 + t^{-5}q^6a^4 + t^{-4}q^{-4}a^8 + 3t^{-4}a^6 + 3t^{-4}q^4a^4 + 4t^{-3}q^{-2}a^6 + 3t^{-3}q^2a^4 + 2t^{-2}q^{-4}a^6 + 4t^{-2}a^4 + t^{-2}q^4a^2 + t^{-1}q^{-6}a^6 + 3t^{-1}q^{-2}a^4 + t^{-1}q^2a^2 + 3q^{-4}a^4 + a^2 + tq^{-6}a^4 + tq^{-2}a^2 + t^2q^{-4}a^2$ \\
$11n_{92}$ & $t^{-5}q^4a^2 + t^{-4}q^2a^2 + t^{-4}q^6 + t^{-3}a^2 + t^{-3}q^4 + t^{-2}q^{-2}a^2 + 2t^{-2}q^2 + t^{-1}q^{-4}a^2 + 2t^{-1} + t^{-1}q^4a^{-2} + 2q^{-2} + 2 + q^2a^{-2} + tq^{-4} + ta^{-2} + tq^2a^{-2} + t^2q^{-6} + t^2q^{-2}a^{-2} + t^3q^{-4}a^{-2} + t^3q^{-2}a^{-2} + t^4a^{-4}$ \\
$11n_{93}$ & $t^{-10}a^{12} + t^{-9}q^2a^{10} + 2t^{-9}q^4a^{10} + 2t^{-8}q^2a^{10} + 2t^{-8}q^6a^8 + t^{-7}q^{-2}a^{10} + 3t^{-7}a^{10} + 3t^{-7}q^4a^8 + 2t^{-6}q^{-2}a^{10} + t^{-6}a^8 + 5t^{-6}q^2a^8 + t^{-6}q^6a^6 + 2t^{-5}q^{-4}a^{10} + 5t^{-5}a^8 + 2t^{-5}q^4a^6 + 5t^{-4}q^{-2}a^8 + 3t^{-4}q^2a^6 + 3t^{-3}q^{-4}a^8 + 3t^{-3}a^6 + 2t^{-2}q^{-6}a^8 + 3t^{-2}q^{-2}a^6 + 2t^{-1}q^{-4}a^6 + q^{-6}a^6$ \\
$11n_{95}$ & $t^{-9}a^{10} + t^{-8}q^2a^8 + t^{-8}q^4a^8 + 2t^{-7}q^2a^8 + t^{-7}q^6a^6 + t^{-6}q^{-2}a^8 + 2t^{-6}a^8 + 2t^{-6}q^4a^6 + 2t^{-5}q^{-2}a^8 + t^{-5}a^6 + 4t^{-5}q^2a^6 + t^{-4}q^{-4}a^8 + 4t^{-4}a^6 + 2t^{-4}q^4a^4 + 4t^{-3}q^{-2}a^6 + 2t^{-3}q^2a^4 + 2t^{-2}q^{-4}a^6 + 3t^{-2}a^4 + t^{-1}q^{-6}a^6 + 2t^{-1}q^{-2}a^4 + 2q^{-4}a^4$ \\
$11n_{96}$ & $t^{-5}q^2a^4 + t^{-4}a^4 + t^{-4}q^4a^2 + t^{-3}q^{-2}a^4 + t^{-3}q^2a^2 + t^{-3}q^4a^2 + 2t^{-2}a^2 + t^{-2}q^2a^2 + t^{-2}q^6 + t^{-1}q^{-2}a^2 + t^{-1}a^2 + t^{-1}q^2 + t^{-1}q^4 + q^{-4}a^2 + q^{-2}a^2 + 2 + 2q^2 + tq^{-4}a^2 + tq^{-2} + 2t + tq^4a^{-2} + 2t^2q^{-2} + t^2q^2a^{-2} + t^3q^{-4} + t^3a^{-2} + t^4q^{-6} + t^4q^{-2}a^{-2} + t^5q^{-4}a^{-2}$ \\
$11n_{97}$ & $t^{-6}q^2a^4 + t^{-5}q^4a^2 + t^{-4}q^{-2}a^4 + 2t^{-3}a^2 + t^{-2}a^2 + t^{-2}q^2 + t^{-1}q^{-4}a^2 + 2t^{-1}q^2 + q^{-2} + 2 + q^4a^{-2} + 2tq^{-2} + tq^2a^{-2} + 3t^2a^{-2} + t^3q^{-2}a^{-2} + t^3q^2a^{-4} + t^4q^{-4}a^{-2} + t^4a^{-4} + t^5q^{-2}a^{-4}$ \\
$11n_{102}$ & $t^{-2}q^2 + t^{-1}q^4a^{-2} + q^{-2} + q^2a^{-2} + 2ta^{-2} + t^2q^{-2}a^{-2} + t^2a^{-2} + t^2q^2a^{-4} + t^3q^{-4}a^{-2} + t^3a^{-4} + t^3q^2a^{-4} + t^4q^{-2}a^{-4} + t^4a^{-4} + t^5q^{-2}a^{-4} + t^5q^2a^{-6} + t^6a^{-6} + t^7q^{-2}a^{-6} + t^8a^{-8}$ \\
$11n_{104}$ & $t^{-9}q^2a^{10} + 2t^{-8}a^{10} + t^{-8}q^4a^8 + t^{-8}q^6a^8 + t^{-7}q^{-2}a^{10} + 2t^{-7}q^2a^8 + t^{-7}q^4a^8 + t^{-7}q^8a^6 + 3t^{-6}a^8 + t^{-6}q^2a^8 + t^{-6}q^6a^6 + 2t^{-5}q^{-2}a^8 + t^{-5}a^8 + 2t^{-5}q^2a^6 + 2t^{-5}q^4a^6 + t^{-4}q^{-4}a^8 + t^{-4}q^{-2}a^8 + 2t^{-4}a^6 + t^{-4}q^2a^6 + t^{-4}q^6a^4 + t^{-3}q^{-4}a^8 + 2t^{-3}q^{-2}a^6 + 2t^{-3}a^6 + t^{-2}q^{-6}a^8 + t^{-2}q^{-2}a^6 + t^{-2}a^4 + t^{-2}q^2a^4 + 2t^{-1}q^{-4}a^6 + q^{-6}a^6 + q^{-2}a^4 + tq^{-8}a^6 + t^2q^{-6}a^4$ \\
$11n_{107}$ & $t^{-8}q^2a^8 + t^{-7}a^8 + t^{-7}q^4a^6 + t^{-7}q^6a^6 + t^{-6}q^{-2}a^8 + t^{-6}q^2a^6 + t^{-6}q^4a^6 + t^{-6}q^8a^4 + 2t^{-5}a^6 + 2t^{-5}q^2a^6 + t^{-5}q^6a^4 + t^{-4}q^{-2}a^6 + t^{-4}a^6 + t^{-4}q^2a^4 + 3t^{-4}q^4a^4 + t^{-3}q^{-4}a^6 + 2t^{-3}q^{-2}a^6 + t^{-3}a^4 + 2t^{-3}q^2a^4 + t^{-3}q^6a^2 + t^{-2}q^{-4}a^6 + t^{-2}q^{-2}a^4 + 3t^{-2}a^4 + t^{-2}q^4a^2 + t^{-1}q^{-6}a^6 + 2t^{-1}q^{-2}a^4 + t^{-1}q^2a^2 + 3q^{-4}a^4 + a^2 + tq^{-6}a^4 + tq^{-2}a^2 + t^2q^{-8}a^4 + t^2q^{-4}a^2 + t^3q^{-6}a^2$ \\
$11n_{111}$ & $t^{-6}q^2a^6 + 2t^{-5}a^6 + t^{-5}q^4a^4 + t^{-4}q^{-2}a^6 + 2t^{-4}q^2a^4 + t^{-4}q^4a^4 + 3t^{-3}a^4 + t^{-3}q^6a^2 + 2t^{-2}q^{-2}a^4 + t^{-2}a^4 + 2t^{-2}q^2a^2 + t^{-1}q^{-4}a^4 + 2t^{-1}a^2 + 2t^{-1}q^2a^2 + q^{-4}a^4 + 2q^{-2}a^2 + q^4 + 2tq^{-2}a^2 + t^2 + t^3q^{-6}a^2 + t^4q^{-4}$ \\
$11n_{116}$ & $t^{-4}q^2a^2 + t^{-3}q^4 + t^{-2}q^{-2}a^2 + 2t^{-1} + 2 + q^2a^{-2} + tq^{-4} + tq^2a^{-2} + t^2q^{-2}a^{-2} + t^2a^{-2} + t^3q^{-2}a^{-2} + t^3q^2a^{-4} + t^4a^{-4} + t^5q^{-2}a^{-4} + t^6a^{-6}$ \\
$11n_{118}$ & $t^{-9}a^{10} + t^{-8}q^2a^8 + t^{-8}q^4a^8 + t^{-7}q^2a^8 + t^{-7}q^6a^6 + t^{-6}q^{-2}a^8 + t^{-6}a^8 + t^{-6}q^4a^6 + t^{-5}q^{-2}a^8 + t^{-5}a^6 + 3t^{-5}q^2a^6 + t^{-4}q^{-4}a^8 + 2t^{-4}a^6 + 2t^{-4}q^4a^4 + 3t^{-3}q^{-2}a^6 + t^{-3}q^2a^4 + t^{-2}q^{-4}a^6 + 2t^{-2}a^4 + t^{-1}q^{-6}a^6 + t^{-1}q^{-2}a^4 + 2q^{-4}a^4$ \\
$11n_{120}$ & $t^{-7}q^2a^6 + t^{-6}a^6 + t^{-6}q^4a^4 + t^{-6}q^6a^4 + t^{-5}q^{-2}a^6 + t^{-5}q^2a^4 + 2t^{-5}q^4a^4 + t^{-5}q^8a^2 + 2t^{-4}a^4 + 3t^{-4}q^2a^4 + 2t^{-4}q^6a^2 + t^{-3}q^{-2}a^4 + 3t^{-3}a^4 + t^{-3}q^2a^2 + 4t^{-3}q^4a^2 + t^{-2}q^{-4}a^4 + 3t^{-2}q^{-2}a^4 + t^{-2}a^2 + 5t^{-2}q^2a^2 + t^{-2}q^6 + 2t^{-1}q^{-4}a^4 + t^{-1}q^{-2}a^2 + 6t^{-1}a^2 + 2t^{-1}q^4 + q^{-6}a^4 + 5q^{-2}a^2 + 3q^2 + 4tq^{-4}a^2 + 2t + 2t^2q^{-6}a^2 + 3t^2q^{-2} + t^3q^{-8}a^2 + 2t^3q^{-4} + t^4q^{-6}$ \\
$11n_{126}$ & $2t^{-10}a^{12} + 2t^{-9}q^2a^{10} + 2t^{-9}q^4a^{10} + t^{-8}a^{10} + t^{-8}q^2a^{10} + 2t^{-8}q^6a^8 + 2t^{-7}q^{-2}a^{10} + 2t^{-7}a^{10} + t^{-7}q^2a^8 + 2t^{-7}q^4a^8 + t^{-6}q^{-2}a^{10} + 2t^{-6}a^8 + 4t^{-6}q^2a^8 + t^{-6}q^6a^6 + 2t^{-5}q^{-4}a^{10} + t^{-5}q^{-2}a^8 + 3t^{-5}a^8 + 2t^{-5}q^4a^6 + 4t^{-4}q^{-2}a^8 + t^{-4}a^6 + 2t^{-4}q^2a^6 + 2t^{-3}q^{-4}a^8 + 2t^{-3}a^6 + 2t^{-2}q^{-6}a^8 + 2t^{-2}q^{-2}a^6 + 2t^{-1}q^{-4}a^6 + q^{-6}a^6$ \\
$11n_{133}$ & $t^{-8}q^2a^8 + 2t^{-7}a^8 + t^{-7}q^4a^6 + t^{-7}q^6a^6 + t^{-6}q^{-2}a^8 + 2t^{-6}q^2a^6 + 2t^{-6}q^4a^6 + t^{-6}q^8a^4 + 3t^{-5}a^6 + 2t^{-5}q^2a^6 + 2t^{-5}q^6a^4 + 2t^{-4}q^{-2}a^6 + 2t^{-4}a^6 + 2t^{-4}q^2a^4 + 3t^{-4}q^4a^4 + t^{-3}q^{-4}a^6 + 2t^{-3}q^{-2}a^6 + 2t^{-3}a^4 + 4t^{-3}q^2a^4 + t^{-3}q^6a^2 + 2t^{-2}q^{-4}a^6 + 2t^{-2}q^{-2}a^4 + 3t^{-2}a^4 + 2t^{-2}q^4a^2 + t^{-1}q^{-6}a^6 + 4t^{-1}q^{-2}a^4 + t^{-1}a^2 + t^{-1}q^2a^2 + 3q^{-4}a^4 + 2a^2 + 2tq^{-6}a^4 + tq^{-2}a^2 + t^2q^{-8}a^4 + 2t^2q^{-4}a^2 + t^3q^{-6}a^2$ \\
$11n_{135}$ & $t^{-8}q^2a^8 + t^{-7}a^8 + t^{-7}q^4a^6 + t^{-6}q^{-2}a^8 + t^{-6}q^2a^6 + t^{-6}q^4a^6 + 2t^{-5}a^6 + t^{-5}q^2a^6 + t^{-5}q^6a^4 + t^{-4}q^{-2}a^6 + t^{-4}a^6 + t^{-4}q^2a^4 + t^{-4}q^4a^4 + t^{-3}q^{-4}a^6 + t^{-3}q^{-2}a^6 + t^{-3}a^4 + 2t^{-3}q^2a^4 + t^{-2}q^{-4}a^6 + t^{-2}q^{-2}a^4 + t^{-2}a^4 + t^{-2}q^4a^2 + 2t^{-1}q^{-2}a^4 + q^{-4}a^4 + a^2 + tq^{-6}a^4 + t^2q^{-4}a^2$ \\
$11n_{138}$ & $t^{-2}q^2a^2 + t^{-1}q^4 + q^{-2}a^2 + 1 + q^2 + 2t + tq^4a^{-2} + t^2q^{-2} + t^2q^2a^{-2} + t^3q^{-4} + 2t^3a^{-2} + t^4q^{-2}a^{-2} + t^4q^2a^{-4} + t^5q^{-4}a^{-2} + t^6q^{-2}a^{-4}$ \\
$11n_{143}$ & $t^{-6}q^4a^4 + t^{-5}q^2a^4 + t^{-5}q^6a^2 + t^{-4}a^4 + t^{-4}q^4a^2 + t^{-3}q^{-2}a^4 + 2t^{-3}q^2a^2 + t^{-2}q^{-4}a^4 + 2t^{-2}a^2 + t^{-2}q^2a^2 + t^{-2}q^4 + 2t^{-1}q^{-2}a^2 + t^{-1}a^2 + t^{-1}q^2 + t^{-1}q^4 + q^{-4}a^2 + q^{-2}a^2 + 2 + q^2 + tq^{-6}a^2 + tq^{-2} + 2t + t^2q^{-4} + t^2q^{-2} + t^2q^2a^{-2} + t^3q^{-4} + t^3a^{-2} + t^4q^{-2}a^{-2}$ \\
$11n_{145}$ & $t^{-5}q^2a^4 + 2t^{-4}a^4 + t^{-4}q^4a^2 + t^{-3}q^{-2}a^4 + 2t^{-3}q^2a^2 + t^{-3}q^4a^2 + 3t^{-2}a^2 + t^{-2}q^6 + 2t^{-1}q^{-2}a^2 + t^{-1}a^2 + 2t^{-1}q^2 + q^{-4}a^2 + 3 + 2q^2 + tq^{-4}a^2 + 2tq^{-2} + tq^4a^{-2} + 2t^2q^{-2} + t^2a^{-2} + t^3a^{-2} + t^4q^{-6} + t^5q^{-4}a^{-2}$ \\
$11n_{147}$ & $t^{-8}q^2a^8 + 2t^{-7}a^8 + t^{-7}q^4a^6 + t^{-7}q^6a^6 + t^{-6}q^{-2}a^8 + 2t^{-6}q^2a^6 + 2t^{-6}q^4a^6 + t^{-6}q^8a^4 + 2t^{-5}a^6 + 3t^{-5}q^2a^6 + 2t^{-5}q^6a^4 + 2t^{-4}q^{-2}a^6 + 2t^{-4}a^6 + t^{-4}q^2a^4 + 4t^{-4}q^4a^4 + t^{-3}q^{-4}a^6 + 3t^{-3}q^{-2}a^6 + 2t^{-3}a^4 + 4t^{-3}q^2a^4 + t^{-3}q^6a^2 + 2t^{-2}q^{-4}a^6 + t^{-2}q^{-2}a^4 + 5t^{-2}a^4 + 2t^{-2}q^4a^2 + t^{-1}q^{-6}a^6 + 4t^{-1}q^{-2}a^4 + 2t^{-1}q^2a^2 + 4q^{-4}a^4 + 2a^2 + 2tq^{-6}a^4 + 2tq^{-2}a^2 + t^2q^{-8}a^4 + 2t^2q^{-4}a^2 + t^3q^{-6}a^2$ \\
$11n_{148}$ & $t^{-7}q^2a^6 + t^{-6}a^6 + t^{-6}q^4a^4 + t^{-6}q^6a^4 + t^{-5}q^{-2}a^6 + t^{-5}q^2a^4 + 3t^{-5}q^4a^4 + t^{-5}q^8a^2 + 2t^{-4}a^4 + 4t^{-4}q^2a^4 + 3t^{-4}q^6a^2 + t^{-3}q^{-2}a^4 + 6t^{-3}a^4 + t^{-3}q^2a^2 + 5t^{-3}q^4a^2 + t^{-2}q^{-4}a^4 + 4t^{-2}q^{-2}a^4 + t^{-2}a^2 + 9t^{-2}q^2a^2 + t^{-2}q^6 + 3t^{-1}q^{-4}a^4 + t^{-1}q^{-2}a^2 + 8t^{-1}a^2 + 3t^{-1}q^4 + q^{-6}a^4 + 9q^{-2}a^2 + 4q^2 + 5tq^{-4}a^2 + 5t + 3t^2q^{-6}a^2 + 4t^2q^{-2} + t^3q^{-8}a^2 + 3t^3q^{-4} + t^4q^{-6}$ \\
$11n_{149}$ & $t^{-8}q^2a^8 + t^{-7}a^8 + t^{-7}q^4a^6 + t^{-7}q^6a^6 + t^{-6}q^{-2}a^8 + t^{-6}q^2a^6 + 2t^{-6}q^4a^6 + t^{-6}q^8a^4 + 2t^{-5}a^6 + 2t^{-5}q^2a^6 + 2t^{-5}q^6a^4 + t^{-4}q^{-2}a^6 + 2t^{-4}a^6 + t^{-4}q^2a^4 + 3t^{-4}q^4a^4 + t^{-3}q^{-4}a^6 + 2t^{-3}q^{-2}a^6 + t^{-3}a^4 + 4t^{-3}q^2a^4 + t^{-3}q^6a^2 + 2t^{-2}q^{-4}a^6 + t^{-2}q^{-2}a^4 + 3t^{-2}a^4 + 2t^{-2}q^4a^2 + t^{-1}q^{-6}a^6 + 4t^{-1}q^{-2}a^4 + t^{-1}q^2a^2 + 3q^{-4}a^4 + 2a^2 + 2tq^{-6}a^4 + tq^{-2}a^2 + t^2q^{-8}a^4 + 2t^2q^{-4}a^2 + t^3q^{-6}a^2$ \\
$11n_{150}$ & $t^{-8}a^8 + t^{-7}q^2a^6 + t^{-7}q^4a^6 + 2t^{-6}q^2a^6 + t^{-6}q^6a^4 + t^{-5}q^{-2}a^6 + 2t^{-5}a^6 + 3t^{-5}q^4a^4 + 2t^{-4}q^{-2}a^6 + t^{-4}a^4 + 6t^{-4}q^2a^4 + t^{-4}q^6a^2 + t^{-3}q^{-4}a^6 + 8t^{-3}a^4 + 4t^{-3}q^4a^2 + 6t^{-2}q^{-2}a^4 + 7t^{-2}q^2a^2 + 3t^{-1}q^{-4}a^4 + 8t^{-1}a^2 + t^{-1}q^4 + q^{-6}a^4 + 7q^{-2}a^2 + 3q^2 + 4tq^{-4}a^2 + 3t + t^2q^{-6}a^2 + 3t^2q^{-2} + t^3q^{-4}$ \\
$11n_{151}$ & $t^{-7}q^4a^6 + 2t^{-6}q^2a^6 + t^{-6}q^6a^4 + 3t^{-5}a^6 + 2t^{-5}q^4a^4 + 2t^{-4}q^{-2}a^6 + 4t^{-4}q^2a^4 + t^{-4}q^4a^4 + t^{-3}q^{-4}a^6 + 5t^{-3}a^4 + t^{-3}q^4a^2 + t^{-3}q^6a^2 + 4t^{-2}q^{-2}a^4 + 2t^{-2}a^4 + 3t^{-2}q^2a^2 + 2t^{-1}q^{-4}a^4 + 3t^{-1}a^2 + 3t^{-1}q^2a^2 + q^{-6}a^4 + q^{-4}a^4 + 3q^{-2}a^2 + q^4 + tq^{-4}a^2 + 3tq^{-2}a^2 + 2t^2 + t^3q^{-6}a^2 + t^4q^{-4}$ \\
$11n_{152}$ & $t^{-7}q^4a^6 + 2t^{-6}q^2a^6 + t^{-6}q^6a^4 + 3t^{-5}a^6 + 2t^{-5}q^4a^4 + 2t^{-4}q^{-2}a^6 + 4t^{-4}q^2a^4 + t^{-4}q^4a^4 + t^{-3}q^{-4}a^6 + 5t^{-3}a^4 + t^{-3}q^4a^2 + t^{-3}q^6a^2 + 4t^{-2}q^{-2}a^4 + 2t^{-2}a^4 + 3t^{-2}q^2a^2 + 2t^{-1}q^{-4}a^4 + 3t^{-1}a^2 + 3t^{-1}q^2a^2 + q^{-6}a^4 + q^{-4}a^4 + 3q^{-2}a^2 + q^4 + tq^{-4}a^2 + 3tq^{-2}a^2 + 2t^2 + t^3q^{-6}a^2 + t^4q^{-4}$ \\
$11n_{153}$ & $t^{-6}q^2a^4 + t^{-5}q^4a^2 + t^{-5}q^6a^2 + t^{-4}q^{-2}a^4 + 2t^{-4}q^4a^2 + t^{-4}q^8 + 2t^{-3}a^2 + 3t^{-3}q^2a^2 + 2t^{-3}q^6 + 4t^{-2}a^2 + t^{-2}q^2 + 4t^{-2}q^4 + t^{-1}q^{-4}a^2 + 3t^{-1}q^{-2}a^2 + 6t^{-1}q^2 + t^{-1}q^6a^{-2} + 2q^{-4}a^2 + q^{-2} + 7 + 2q^4a^{-2} + tq^{-6}a^2 + 6tq^{-2} + 3tq^2a^{-2} + 4t^2q^{-4} + 4t^2a^{-2} + 2t^3q^{-6} + 3t^3q^{-2}a^{-2} + t^4q^{-8} + 2t^4q^{-4}a^{-2} + t^5q^{-6}a^{-2}$ \\
$11n_{155}$ & $t^{-6}a^6 + t^{-5}q^2a^4 + t^{-5}q^4a^4 + t^{-4}a^4 + 2t^{-4}q^2a^4 + t^{-4}q^6a^2 + t^{-3}q^{-2}a^4 + 2t^{-3}a^4 + t^{-3}q^2a^2 + 3t^{-3}q^4a^2 + 2t^{-2}q^{-2}a^4 + t^{-2}a^2 + 5t^{-2}q^2a^2 + t^{-2}q^6 + t^{-1}q^{-4}a^4 + t^{-1}q^{-2}a^2 + 5t^{-1}a^2 + 3t^{-1}q^4 + 5q^{-2}a^2 + 1 + 4q^2 + 3tq^{-4}a^2 + 5t + tq^4a^{-2} + t^2q^{-6}a^2 + 4t^2q^{-2} + 2t^2q^2a^{-2} + 3t^3q^{-4} + t^3a^{-2} + t^4q^{-6} + 2t^4q^{-2}a^{-2} + t^5q^{-4}a^{-2}$ \\
$11n_{158}$ & $t^{-8}q^2a^8 + t^{-7}a^8 + t^{-7}q^4a^6 + t^{-7}q^6a^6 + t^{-6}q^{-2}a^8 + t^{-6}q^2a^6 + 2t^{-6}q^4a^6 + t^{-6}q^8a^4 + 2t^{-5}a^6 + 3t^{-5}q^2a^6 + 2t^{-5}q^6a^4 + t^{-4}q^{-2}a^6 + 3t^{-4}a^6 + t^{-4}q^2a^4 + 4t^{-4}q^4a^4 + t^{-3}q^{-4}a^6 + 3t^{-3}q^{-2}a^6 + t^{-3}a^4 + 5t^{-3}q^2a^4 + t^{-3}q^6a^2 + 2t^{-2}q^{-4}a^6 + t^{-2}q^{-2}a^4 + 5t^{-2}a^4 + 2t^{-2}q^4a^2 + t^{-1}q^{-6}a^6 + 5t^{-1}q^{-2}a^4 + 2t^{-1}q^2a^2 + 4q^{-4}a^4 + 3a^2 + 2tq^{-6}a^4 + 2tq^{-2}a^2 + t^2q^{-8}a^4 + 2t^2q^{-4}a^2 + t^3q^{-6}a^2$ \\
$11n_{161}$ & $t^{-8}a^8 + t^{-7}q^2a^6 + t^{-7}q^4a^6 + 2t^{-6}q^2a^6 + t^{-6}q^6a^4 + t^{-5}q^{-2}a^6 + 2t^{-5}a^6 + 3t^{-5}q^4a^4 + 2t^{-4}q^{-2}a^6 + t^{-4}a^4 + 5t^{-4}q^2a^4 + t^{-4}q^6a^2 + t^{-3}q^{-4}a^6 + 7t^{-3}a^4 + 3t^{-3}q^4a^2 + 5t^{-2}q^{-2}a^4 + 6t^{-2}q^2a^2 + 3t^{-1}q^{-4}a^4 + 6t^{-1}a^2 + t^{-1}q^4 + q^{-6}a^4 + 6q^{-2}a^2 + 2q^2 + 3tq^{-4}a^2 + 2t + t^2q^{-6}a^2 + 2t^2q^{-2} + t^3q^{-4}$ \\
$11n_{166}$ & $t^{-7}q^2a^6 + t^{-6}a^6 + t^{-6}q^4a^4 + t^{-6}q^6a^4 + t^{-5}q^{-2}a^6 + t^{-5}q^2a^4 + 2t^{-5}q^4a^4 + t^{-5}q^8a^2 + 2t^{-4}a^4 + 4t^{-4}q^2a^4 + 2t^{-4}q^6a^2 + t^{-3}q^{-2}a^4 + 4t^{-3}a^4 + t^{-3}q^2a^2 + 5t^{-3}q^4a^2 + t^{-2}q^{-4}a^4 + 4t^{-2}q^{-2}a^4 + t^{-2}a^2 + 6t^{-2}q^2a^2 + t^{-2}q^6 + 2t^{-1}q^{-4}a^4 + t^{-1}q^{-2}a^2 + 8t^{-1}a^2 + 2t^{-1}q^4 + q^{-6}a^4 + 6q^{-2}a^2 + 4q^2 + 5tq^{-4}a^2 + 3t + 2t^2q^{-6}a^2 + 4t^2q^{-2} + t^3q^{-8}a^2 + 2t^3q^{-4} + t^4q^{-6}$ \\
$11n_{169}$ & $t^{-10}a^{12} + t^{-9}q^2a^{10} + 2t^{-9}q^4a^{10} + t^{-8}q^2a^{10} + 2t^{-8}q^6a^8 + t^{-7}q^{-2}a^{10} + 2t^{-7}a^{10} + 2t^{-7}q^4a^8 + t^{-6}q^{-2}a^{10} + t^{-6}a^8 + 4t^{-6}q^2a^8 + t^{-6}q^6a^6 + 2t^{-5}q^{-4}a^{10} + 3t^{-5}a^8 + 2t^{-5}q^4a^6 + 4t^{-4}q^{-2}a^8 + 2t^{-4}q^2a^6 + 2t^{-3}q^{-4}a^8 + 2t^{-3}a^6 + 2t^{-2}q^{-6}a^8 + 2t^{-2}q^{-2}a^6 + 2t^{-1}q^{-4}a^6 + q^{-6}a^6$ \\
$11n_{173}$ & $t^{-8}q^2a^8 + 2t^{-7}a^8 + t^{-7}q^4a^6 + t^{-7}q^6a^6 + t^{-6}q^{-2}a^8 + 2t^{-6}q^2a^6 + 2t^{-6}q^4a^6 + t^{-6}q^8a^4 + 2t^{-5}a^6 + 4t^{-5}q^2a^6 + 2t^{-5}q^6a^4 + 2t^{-4}q^{-2}a^6 + 2t^{-4}a^6 + t^{-4}q^2a^4 + 5t^{-4}q^4a^4 + t^{-3}q^{-4}a^6 + 4t^{-3}q^{-2}a^6 + 2t^{-3}a^4 + 4t^{-3}q^2a^4 + t^{-3}q^6a^2 + 2t^{-2}q^{-4}a^6 + t^{-2}q^{-2}a^4 + 7t^{-2}a^4 + 2t^{-2}q^4a^2 + t^{-1}q^{-6}a^6 + 4t^{-1}q^{-2}a^4 + 3t^{-1}q^2a^2 + 5q^{-4}a^4 + 2a^2 + 2tq^{-6}a^4 + 3tq^{-2}a^2 + t^2q^{-8}a^4 + 2t^2q^{-4}a^2 + t^3q^{-6}a^2$ \\
$11n_{175}$ & $t^{-9}a^{10} + t^{-8}q^2a^8 + t^{-8}q^4a^8 + 2t^{-7}q^2a^8 + t^{-7}q^6a^6 + t^{-6}q^{-2}a^8 + 2t^{-6}a^8 + 3t^{-6}q^4a^6 + 2t^{-5}q^{-2}a^8 + t^{-5}a^6 + 6t^{-5}q^2a^6 + t^{-5}q^6a^4 + t^{-4}q^{-4}a^8 + 6t^{-4}a^6 + 4t^{-4}q^4a^4 + 6t^{-3}q^{-2}a^6 + 5t^{-3}q^2a^4 + 3t^{-2}q^{-4}a^6 + 7t^{-2}a^4 + t^{-2}q^4a^2 + t^{-1}q^{-6}a^6 + 5t^{-1}q^{-2}a^4 + 2t^{-1}q^2a^2 + 4q^{-4}a^4 + 2a^2 + tq^{-6}a^4 + 2tq^{-2}a^2 + t^2q^{-4}a^2$ \\
$11n_{177}$ & $t^{-7}q^2a^6 + t^{-6}a^6 + t^{-6}q^4a^4 + t^{-6}q^6a^4 + t^{-5}q^{-2}a^6 + t^{-5}q^2a^4 + 3t^{-5}q^4a^4 + t^{-5}q^8a^2 + 2t^{-4}a^4 + 5t^{-4}q^2a^4 + 3t^{-4}q^6a^2 + t^{-3}q^{-2}a^4 + 6t^{-3}a^4 + t^{-3}q^2a^2 + 6t^{-3}q^4a^2 + t^{-2}q^{-4}a^4 + 5t^{-2}q^{-2}a^4 + t^{-2}a^2 + 9t^{-2}q^2a^2 + t^{-2}q^6 + 3t^{-1}q^{-4}a^4 + t^{-1}q^{-2}a^2 + 10t^{-1}a^2 + 3t^{-1}q^4 + q^{-6}a^4 + 9q^{-2}a^2 + 5q^2 + 6tq^{-4}a^2 + 5t + 3t^2q^{-6}a^2 + 5t^2q^{-2} + t^3q^{-8}a^2 + 3t^3q^{-4} + t^4q^{-6}$ \\
$11n_{182}$ & $t^{-5}q^6a^2 + 3t^{-4}q^4a^2 + t^{-4}q^8 + 5t^{-3}q^2a^2 + 3t^{-3}q^6 + 7t^{-2}a^2 + 6t^{-2}q^4 + 5t^{-1}q^{-2}a^2 + 10t^{-1}q^2 + t^{-1}q^6a^{-2} + 3q^{-4}a^2 + 11 + q^2 + 3q^4a^{-2} + tq^{-6}a^2 + 10tq^{-2} + 5tq^2a^{-2} + tq^4a^{-2} + 6t^2q^{-4} + t^2q^{-2} + 7t^2a^{-2} + 3t^3q^{-6} + 5t^3q^{-2}a^{-2} + 2t^3a^{-2} + t^4q^{-8} + 3t^4q^{-4}a^{-2} + t^4q^2a^{-4} + t^5q^{-6}a^{-2} + t^5q^{-4}a^{-2} + t^6q^{-2}a^{-4}$ \\
$11n_{183}$ & $t^{-10}a^{12} + 3t^{-9}q^2a^{10} + 2t^{-8}a^{10} + 2t^{-8}q^4a^8 + 3t^{-7}q^{-2}a^{10} + 2t^{-7}q^2a^8 + t^{-7}q^4a^8 + 5t^{-6}a^8 + t^{-6}q^6a^6 + 2t^{-5}q^{-2}a^8 + t^{-5}a^8 + 2t^{-5}q^2a^6 + 2t^{-4}q^{-4}a^8 + 2t^{-4}a^6 + t^{-4}q^2a^6 + t^{-3}q^{-4}a^8 + 2t^{-3}q^{-2}a^6 + t^{-2}q^{-2}a^6 + q^{-6}a^6$ \\
\end{longtable}

\clearpage
\restoregeometry
\end{document}